\documentclass[notitlepage,11pt]{article}

\usepackage[numbers,sort&compress]{natbib}

\usepackage{amsmath,amsthm,amssymb,amsfonts}
\usepackage[margin=1in]{geometry} 
\usepackage{graphicx}
\usepackage{commath,mathtools,bbm}
\usepackage{dsfont}
\usepackage{xargs}  
\usepackage{xifthen}
\usepackage{subfigure}
\usepackage{url}
\hypersetup{
colorlinks=true,
citecolor=blue
}
\usepackage{enumerate}
\usepackage[normalem]{ulem}

\usepackage{tikz}
\usetikzlibrary{positioning,shapes,fit,arrows}

\usepackage{pst-solides3d}
\usepackage{pstricks} 
\usepackage{pst-all}
\usepackage{pst-eucl}

\definecolor{myblue}{RGB}{56,94,141}

\DeclareMathOperator{\ext}{vert}

\DeclareMathOperator{\co}{conv}

\DeclareMathOperator{\proj}{Proj}

\DeclareMathOperator{\degree}{deg}

\newtheorem{theorem}{Theorem}[section]
\newtheorem{proposition}{Proposition}[section]
\newtheorem{corollary}{Corollary}[section]

\newtheorem{lemma}{Lemma}[section]
\newtheorem{observation}{Observation}[section]
\newtheorem{claim}{Claim}[section]

\theoremstyle{remark}
\newtheorem{remark}{Remark}

\theoremstyle{definition}

\newcommand{\V}{V}
\newcommand{\fancyV}{\mathcal{V}}
\newcommand{\fancyH}{\mathcal{H}}
\newcommand{\conv}[1]{\co #1}
\renewcommand{\P}{\mathcal{P}}
\newcommand{\pink}{H_0}
\newcommand{\grlex}{\preceq}
\newcommand{\lex}{\le_{\mathrm{lex}}}
\newcommand{\lexge}{\ge_{\mathrm{lex}}}
\renewcommand{\preceq}{\preccurlyeq}

\renewcommand{\b}{b}
\newcommand{\bb}{\tilde{b}}
\newcommand{\y}{w}
\newcommand{\five}[1]{u^{#1}}
\newcommand{\six}[2]{v^{#1,#2}}
\newcommand{\n}{d}
\newcommand{\G}[1]{G(#1)}

\newcommand{\ints}{\mathbb{Z}}
\newcommand{\real}{\mathbb{R}}
\newcommand{\zerovec}{\boldsymbol{0}} 
\newcommand{\onevec}[1][]{%
\ifthenelse{\isempty{#1}}%
{\mathbf{\mathds{1}}}
{\mathbf{e}_{#1}}
}

\newcommandx{\N}[2][1=\theta,2=\P,usedefault]{\mathcal{N}_{#2}(#1)} 

\newcommand{\bool}[1][]{%
\ifthenelse{\isempty{#1}}%
{\{0,1\}}
{\{0,1\}^{#1}}
}

\newcommandx{\C}[2][1=,2=\P,usedefault]{\mathcal{C}_{#2}(#1)}
\newcommand{\X}{X}
\newcommand{\VV}{\ext(\X)}
\newcommand{\x}{x}

\newcommand{\Q}{\mathcal{Q}}
\newcommand{\ub}[1]{\bar{u}^{#1}}
\newcommand{\vb}[2]{\bar{v}^{#1,#2}}
\newcommand{\grevlex}{\preceq}
\newcommand{\Dant}[3]{#1 \text{ is a } (#2,#3)\text{-Dantzig figure}}

\newcommand{\myred}[1]{#1}

\renewcommand{\mid}{\,\colon}

\newcommand{\keywords}[1]{~\smallskip\\
\noindent\textbf{Keywords.} #1}
\newcommand{\subjclass}[1]{~\smallskip\\
\noindent\textbf{MSC 2010.} #1}

\title{On Dantzig figures from \myred{graded} lexicographic orders}
\author{Akshay Gupte\renewcommand{\footnotemark}{}\footnote{Department of Mathematical Sciences, Clemson University}\footnote{\emph{Email addresses}: \texttt{\{agupte,spoznan\}@clemson.edu}} \and Svetlana Poznanovi\'{c}}
\date{}

\begin{document}
\maketitle

\begin{abstract}
We construct two families of Dantzig figures, which are $(d,2d)$-polytopes with an antipodal vertex pair, from convex hulls of initial subsets for the graded lexicographic (grlex) and graded reverse lexicographic (grevlex) orders on $\mathbb{Z}^{d}_{\geq 0}$. These two polytopes have the same number of vertices, $\mathcal{O}(d^{2})$, \myred{and the same number of edges, $\mathcal{O}(d^{3})$}, but are not combinatorially equivalent. We provide an explicit description of the vertices and the facets for both families and describe their graphs along with analyzing their basic properties such as the radius, diameter, existence of Hamiltonian circuits, and chromatic number. 
Moreover, we also \myred{analyze the edge expansions of these graphs}.
\end{abstract}

{
\renewcommand{\grlex}{\preceq_{\mathrm{gr}}}
\renewcommand{\grevlex}{\preceq_{\mathrm{grev}}}
\newcommand{\Plex}{P^{\mathrm{lex}}}

\section{Introduction}
A $d$-polytope is a bounded convex polyhedron whose affine dimension is equal to $d$. Equivalently, a $d$-polytope is the convex hull of finitely many points, exactly $d+1$ of which are affinely independent. It is simple if every vertex is defined by exactly $d$ facets, or equivalently, has exactly $d$ neighboring vertices; otherwise it is non-simple. A $d$-polytope with $n$ facets is referred to as a $(d,n)$-polytope. When $n=2d$, we have a $(d,2d)$-polytope. A $(d,2d)$-polytope $\X$ is said to be a \emph{Dantzig figure} generated by distinct vertices $u$ and $v$ if $u$ and $v$ do not share a common facet. In this case we say that $\Dant{\X}{u}{v}$.  Thus for a Dantzig figure, exactly $d$ distinct facets are incident to each of $u$ and $v$, and every facet contains exactly one of $u$ or $v$. This also means that both $u$ and $v$ have exactly $d$ neighboring vertices. Since $u$ and $v$ do not lie on the same facet, they are called an antipodal vertex pair, \myred{and a figure may have multiple such pairs}. Trivial examples include the hypercube and the simplicial bipyramid.

\myred{Dantzig figures were introduced by \citet{dantzig19648} in the context of the Hirsch conjecture on combinatorial  diameter of $(d,n)$-polytopes, and gained prominence after it was shown \citep{klee1967thed} that this conjecture would be true for all polytopes if and only if it was true for simple Dantzig figures. Although the Hirsch conjecture was disproved recently  \citep{santos2012counterexample}, diameters of special-structured polytopes have always been, and continue to be, the topic of study in literature \citep{de2014combinatorics,Borgwardt2017,padberg1974travelling,rispoli1998bound,bremner2013more}. Besides the connection to combinatorial diameter, Dantzig figures are also important from the perspective of them being polytopes with not too many facets, i.e., belonging to the family of $(d,kd)$-polytopes for some small constant $k$. Polytopes with few facets, where few facets generally means  $(d,d+k)$-polytopes, have been studied recently for their combinatorial properties \citep{padrol2016extension,padrol2016polytopes}. An important question in polyhedral combinatorics is the identification of different combinatorial types of a particular family of polytopes. This has been answered for $(d,d+k)$-polytopes for small $k$ (typically $k\le 6$) \citep{grunbaum2003convex,bremner2011edge}. On the other hand, this question has gone largely unanswered for $(d,kd)$-polytopes, and their explicit construction has received limited attention. There are results, though, showing how some $(d,kd)$-polytopes arise from a term order. Given $\theta,u\in\ints^{d}$ with $\zerovec\le\theta\le u$ and the lexicographic (lex) order $\lex$ on $\ints^{d}$, the lex polytope is  
\[
\Plex:=\conv{\{x\in\ints^{d}\mid \zerovec \le x \lex \theta,\ x \le u\}}.
\] 
Note that the upper bound $x\le u$ is necessary to obtain a polytope because the lex constraint $x\lex y$ over the reals defines a neither open nor closed convex cone. The $0\backslash 1$ polytope $\Plex$ (i.e., with  $u=\onevec$) was shown to be a $(d,3d)$-polytope separately by \citep{laurent1992characterization,gillmann2006revlex}. This was later generalized to arbitrary integral $u$ by \citep{lexico}, who also showed that the polytope $\conv{\{x\in[\zerovec,u]\cap\ints^{d}\mid \gamma \lex x \lex \theta, x \le u\}}$ defined by one $\lex$ and one $\lexge$ order is a $(d,4d)$-polytope. Thus, lex polytopes are $(d,kd)$-polytopes\footnote{Strictly speaking, lex polytopes are, in general, $(d,kd-\epsilon)$-polytopes for $\epsilon\in\{1,2\}$.} for $k\in\{3,4\}$. To the best of our knowledge, explicit ways of constructing nontrivial Dantzig figures, either simple or non-simple, for arbitrary $d$ are unknown.}

\myred{Besides identifying families of $(d,kd)$-polytopes, term orders are also helpful in solving mixed-integer optimization problems. The lex order has been used for breaking symmetry in integer programs \citep{margot2010symmetry}, which has subsequently led to polyhedral studies of associated polytopes \citep{hojny2017polytopes,kaibel2008packing}. Another place were lex-ordered sets and the inequalities defining their convex hull appear is in reformulations of mixed-integer problems \citep{siampaper,lexapprox}. A third application of term orders is their use in strengthening cutting planes for separating a fractional point in branch-and-cut algorithms \citep{twotermlex}. This can be explained briefly as follows. Let $X_{I} = \{(x,y)\in\ints^{n}\times\real^{m}\mid Ax + By \le b\}$ be a mixed-integer feasible region and let $(x^{\ast},y^{\ast})$, with $x^{\ast}\notin\ints^{n}$, be optimal to the linear programming relaxation $X$. There are many well-known techniques \citep[cf.][]{conforti2014book} for finding a hyperplane $\alpha x + \beta y \le \alpha_{0}$ that separates $(x^{\ast},y^{\ast})$ from $X_{I}$, one of the most powerful of them being split cuts. A split cut is obtained by first finding $(\pi,\pi_{0})\in\ints^{n+1}$ such that $\pi_{0} < \pi x^{\ast} < \pi_{0}+1$ and then solving a cut-generating linear program to find a valid inequality to the disjunction $(X\cap\{x\mid \pi x \le \pi_{0} \})\cup (X\cap\{x\mid \pi x \ge \pi_{0}+1 \})$. This inequality can be further strengthened using a term order $\preceq$. Let $\tilde{x}$ be the largest (under $\preceq$) point in $\{x\in\ints^{n}\mid \pi x \le \pi_{0}\}$ such that $(\tilde{x},y)\in X_{I}$ for some $y$. Similarly, let $\hat{x}$ be the smallest such point in  $\{x\in\ints^{n}\mid \pi x \ge \pi_{0}+1\}$. If $\mathcal{C}$ (resp. $\mathcal{D}$) is the polytope defined as the convex hull of all integral points less (resp. greater) than or equal to $\tilde{x}$ (resp. $\hat{x}$), then the disjunction $(X\cap\mathcal{C})\cup(X\cap\mathcal{D})$ can be used to separate a cutting plane that is at least as strong as the one obtained from the split disjunction $\{\pi x \le \pi_{0}\} \cup \{\pi x \ge \pi_{0}+1 \}$. This approach relies on having a complete facet description of polytopes $\mathcal{C}$ and $\mathcal{D}$ that arise from term orders.}

\paragraph{Our Results.} 
We construct two combinatorial types of non-simple $d$-dimensional Dantzig  figures, for any $d\ge 3$,  using two term orders related to the lex order. Thus, we not only advance the study of polytopes arising from term orders but also provide a constructive characterization for some Dantzig figures. \myred{Furthermore, our polytopes fit in a small grid ($\mathcal{O}(d^{2})$ vertices fit in a grid of size $\mathcal{O}(d)$), a class of polytopes in which interesting examples are  relatively scarce.}  

The polytopes we construct are defined by the graded lex (grlex) order ($\grlex$) and the graded reverse lex (grevlex) order ($\grevlex$). 
Given $d\ge 3$ and $\theta\in\ints^{d}_{+}$  with $\theta\ge\onevec$, the grlex and grevlex polytopes are, respectively, 
\begin{equation}\label{eq:defPQ}
\P := \conv{\{x \in \ints^{d}\colon \zerovec \le  x \grlex \theta \}}, \qquad \Q := \conv{\{x \in \ints^{d}\colon \zerovec \le  x \grevlex \theta \}}.
\end{equation}
\myred{From now we assume that  $\theta\ge\onevec$ and $\n \geq 3$. We don't consider the case $\n=2$ because in this case $\P$ and $\Q$ are just quadrilaterals.} Our consideration of these lattice polytopes is motivated by lex polytopes being $(d,kd)$-polytopes for $k\in\{3,4\}$ \myred{and polytopes from term orders being useful for mixed-integer optimization, as mentioned earlier}. Also, note that a projection of $\P$ (or $\Q$) yields the lex polytope over a integral simplex (see Remark~\ref{rem:projection}). 

\myred{We find  the $\fancyV$- and $\fancyH$-representations of these polytopes. The $\fancyH$-representations are obtained using a conic characterization that we develop for arbitrary polytopes. We then characterize the facet-vertex incidence for these polytopes. Based on this, we find that the face lattices of these polytopes are independent of $\theta$: for any $\P$ and $\P^{\prime}$ corresponding to $\theta,\theta^{\prime}>\onevec$, we have $\P\cong\P^{\prime}$ and for any $\Q$ and $\Q^{\prime}$ corresponding to $\theta,\theta^{\prime}\ge\onevec$, we have $\Q\cong\Q^{\prime}$}.

\myred{
The facet-vertex incidence then reveals that $\P$ and $\Q$ are Dantzig figures for all $d$ and $\theta$. \myred{Specifically, we show that $\P=\Dant{\P}{\zerovec}{\theta}$ and $(\zerovec,\theta)$ is the only antipodal vertex pair of  $\P$ (Theorem~\ref{thm:dantzig}). Also, $\Q=\Dant{\Q}{\zerovec}{\theta}$ and $(\zerovec,\theta)$ is the only antipodal vertex pair of  $\Q$ when $\n \geq 4$ (Theorem~\ref{thm:dantzig2}).}

As the polytopes under study are Dantzig figures, which were introduced by Dantzig in relation to the problem of bounding the diameter of polytopes,  it is natural to ask whether our polytopes have a large diameter. This is also interesting from the aspect of the question what is the largest diameter of lattice polytopes whose vertex coordinates are integers between $0$ and $k$. The upper bound for $d$-polytopes was first shown to be $kd$~\cite{kleinschmidt1992diameter} and was recently improved  in~\cite{del2016diameter, deza2016improved}. A lower bound was given in~\cite{deza2017primitive}. We give a complete description of the graphs of the polytopes, $\G{\P}$ and $\G{\Q}$, and show that they  have constant and small diameters.    Interestingly,  $\G{\P}$ and $\G{\Q}$, have not only the same number of vertices, but also the same number of edges $\mathcal{O}(d^{3})$. However, the graphs are not isomorphic for $\theta > \onevec$, meaning that $\P\not\cong\Q$ in general.




A graph is said to have good expansion properties if, roughly speaking, it is sparse but has high connectivity, which is quantified in terms of edge expansion of at least 1. Bounding the edge expansion of graphs of polytopes is of significant interest due to its importance in studying random walks on such graphs and therefore has received much attention \citep{kaibel2001expansion}. It was shown in~\cite{gillmann2006revlex}, that $0\backslash 1$ $\Plex$, even though sparse, has edge expansion at least 1. Since the graphs $G(\P)$ and $G(\Q)$ are sparse with  average degree  $\mathcal{O}(d)$, we analyzed their edge expansion.   We show that $\G{\P}$ lies on the threshold for polytopes with good and poor expansion properties  with edge expansion $h(G(\P)) =1$ (in general, computing the edge expansion for general graphs is NP-hard~\citep[Theorem 2,][]{kaibel2001expansion}). Numerical results show that $h(G(\Q))$ depends on $\n$ but we believe that it is also at least 1. }


\paragraph{Notation.}
The vector of all zeros is $\zerovec$, the vector of all ones is $\onevec$, and the $i^{th}$ unit coordinate vector is $\onevec[i]$. Let $\lex$ denote the lexicographic monomial order. For $x,y\in\ints^{d}$, we say $x\lex y$ if either $x=y$ or there exists some $i$ with $x_{i}<y_{i}$ and $x_{k}=y_{k}$ for all $k > i$.\footnote{Our right-to-left order of coordinate comparison here is opposite to the left-to-right order generally used in literature, but this is immaterial \myred{up to} permuting the  variables.} The graded lex (grlex) and graded reverse lex (grevlex) monomial orders are denoted as $\grlex$ and $\grevlex$, respectively, and defined as follows: 
\begin{enumerate}
\item $x\grlex y$ if either $\sum_{i=1}^{d}x_{i} < \sum_{i=1}^{d}y_{i}$, or $\sum_{i=1}^{d}x_{i} = \sum_{i=1}^{d}y_{i}$ and $x\lex y$,
\item $x\grevlex y$ if either $\sum_{i=1}^{d}x_{i} < \sum_{i=1}^{d}y_{i}$, or $\sum_{i=1}^{d}x_{i} = \sum_{i=1}^{d}y_{i}$ and $x\lexge y$.
\end{enumerate}
Denoting
\begin{subequations} 
\begin{equation} \label{b-bk}
\b := \sum_{i=1}^{\n}\theta_{i}, \qquad \bb_{k} := \sum_{i=1}^{k}\theta_{i} = \b - \sum_{i=k+1}^{d}\theta_{i} \quad 1 \le k\le d, \qquad \pink := \left\{x\in\real^{d}\colon \sum_{i=1}^{d}x_{i}=\b \right\}, 
\end{equation}
as the total and partial sums of $\theta$ and the grading hyperplane, it is clear that
\begin{align}
\P = \conv{ \left(\left\{x\in\real^{d}_{+}\colon \sum_{i=1}^{d}x_{i} \le \b-1\right\} \, \bigcup \,  \conv{ \left\{x\in\ints^{d}_{+}\colon \sum_{i=1}^{d}x_{i} = \b, x \lex \theta \right\}} \right)}\label{eq:Punion}\\
\Q = \conv{ \left(\left\{x\in\real^{d}_{+}\colon \sum_{i=1}^{d}x_{i} \le \b-1\right\} \, \bigcup \,  \conv{ \left\{x\in\ints^{d}_{+}\colon \sum_{i=1}^{d}x_{i} = \b, x \lexge \theta \right\}} \right)}\label{eq:Qunion}.
\end{align}
\end{subequations}
Both $\P$ and $\Q$ are $d$-polytopes since they contain the standard simplex. It is easy to verify that $\P\cap\pink\subsetneq\pink\cap\conv{\{x\in\ints^{d}\mid \zerovec \le x \lex \theta \}}$. Similarly for $\Q$. Thus $\fancyH$-representations of $\P$ and $\Q$ are not a trivial implication of the known results for lex polytopes. 

Figure~\ref{fig:bliblablup} illustrates these polytopes for $d=3$ for \myred{$\theta = (2,2,2)$}. Both have 7 vertices, 11 edges, and 6 facets, but they are not isomorphic because $\P$ has one pentagonal, two quadrilateral, and three triangular facets whereas $\Q$ has two triangular and four quadrilateral facets. \myred{As we will see, the face lattices of $\P$ and $\Q$ are independent of the actual value of $\theta$ when $\theta >1$.}

\begin{figure}
\centering
\includegraphics{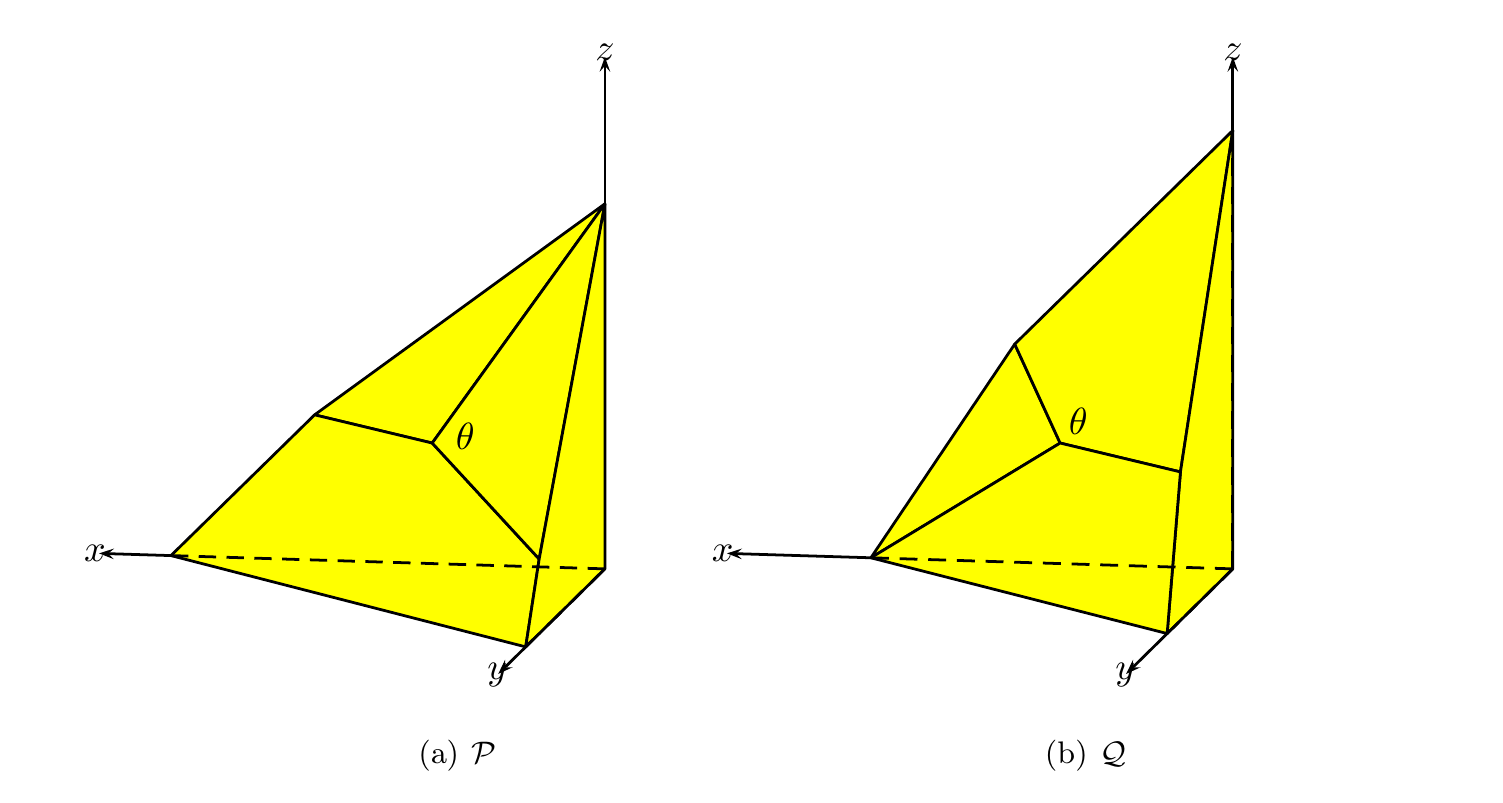}
%
\label{fig:bliblablup}
\caption{The grlex and grevlex polytopes in $d=3$ defined by \myred{$\theta=(2,2,2)$}.}
\end{figure}

%
%
%

\begin{remark}\label{rem:projection}
$\P$ (resp. $\Q$) yields the convex hull of all the integral vectors that belong to a standard integral simplex and are lexicographically smaller (resp. greater) than a fixed integer vector. In particular, denoting $\tilde{x} = (x_{2},\dots,x_{d})$ and $\tilde{\theta} = (\theta_{2},\dots,\theta_{d})$, we have \[\proj_{\tilde{x}}(\P\cap\pink) = \conv{\left\{\tilde{x}\in\ints^{d-1}_{+}\colon \sum_{i=2}^{d}x_{i} \le \b, \tilde{x} \lex\tilde{\theta}\right\}}\] and \[\proj_{\tilde{x}}(\Q\cap\pink) = \conv{\left\{\tilde{x}\in\ints^{d-1}_{+}\colon \sum_{i=2}^{d}x_{i} \le \b, \tilde{x} \lexge\tilde{\theta}\right\}}.\]
\end{remark}



\paragraph{Outline.}
We begin by providing in Section \ref{sec:cone} a conic representation of arbitrary polytopes. This result implies that any Dantzig figure is equal to the intersection of two polyhedral cones, which is obviously the minimum number of cones required to represent any polytope. We use this later to obtain the $\fancyH$-representation of the polytopes $\P$ and $\Q$. The rest of the paper is divided into two parts. Section~\ref{sec:grlex} analyzes the grlex polytope $\P$ and Section~\ref{sec:grevlex} analyzes the grevlex polytope $\Q$. For each polytope, we show it is a non-simple Dantzig figure generated by $\zerovec$ and $\theta$ and identify all its $\mathcal{O}(d^{2})$ vertices and $2d$ facet-defining inequalities. Although $\P$ and $\Q$ appear closely related by definition, they are combinatorially not equivalent, as seen in Figure~\ref{fig:bliblablup}. This necessitates separate proofs, especially for showing the Dantzig figure property in Theorems~\ref{thm:dantzig} and \ref{thm:dantzig2}, but we condense our arguments whenever possible. We describe $\G{\P}$ and $\G{\Q}$, the graphs of these polytopes, and their basic properties, including diameter, in Section~\ref{sec:G} and Section~\ref{sec:G2}. 

}

\section{Conic characterization of polytopes}\label{sec:cone}
\renewcommand{\d}{d}

Let $\X\subseteq\real^{n}$ be a $\d$-polytope with set of vertices  $\VV$. For every $v\in\VV$, $\N[v][\X]$ denotes the set of vertices adjacent to $v$. Recall that two vertices of a $d$-polytope are adjacent if and only if there are at least $d-1$ facets that contain both the vertices. The tangent cone at a vertex $v$ (also referred to as a vertex cone)  is defined as
\begin{equation}
\C[v][\X] := v + \left\{\sum_{\x\in\N[v][\X]}\alpha_{\x}(\x - v)\colon \alpha \ge \zerovec \right\} = v + \operatorname{cone}\{\x-v \}_{\x\in\N[v][\X]}.
\end{equation}
By construction, the dimension of this cone cannot be greater than the dimension of $\X$. Observe that 
\begin{equation}\label{eq:cone}
\X\subseteq\C[v][\X]\quad \forall v\in\VV.
\end{equation}
This can be argued as follows.\footnote{A different proof is given in \citet[Lemma 3.6]{ziegler1995lectures}.} Let $\X=\{x\in\real^{n}\mid Ax \ge b\}$ be an $\fancyH$-representation of $\X$ for some $A\in\real^{m\times n}, b \in\real^{m}$. A basis of $\X$ is a $n$-subset of $[m]$ such that the rows of $A$ indexed by this subset are linearly independent. Consider any $v\in\VV$. By the equivalence of vertices and basic feasible solutions of a polyhedron, there exists some basis $B$ such that $v$ is the unique solution to  the linear system $a_{i}x = b_{i}$ for $i\in B$, where $a_{i}$ is the $i^{th}$ row of $A$. Let $B^{\prime} := \{i\mid a_{i}v = b_{i}\}$; clearly $B^{\prime}\supseteq B$ with the inclusion being strict if and only if $v$ is a degenerate vertex. It is easy to argue then that the tangent cone at $v$ can be represented as 
\begin{equation}\label{eq:cone2}
\C[v][\X] = v + \{y\colon a_{i}y \ge 0, \ i\in B^{\prime}\} .
\end{equation}
Now for any $x$ with $Ax\ge b$, we have $x-v$ satisfying $a_{i}(x-v)\ge 0$ for all $i\in B^{\prime}$, and therefore $x\in\C[v][\X]$.

Equation~\eqref{eq:cone} implies two things. First that the affine dimension of $\C[v][\X]$ is equal to $d$. Secondly, it leads to the inclusion $\X\subseteq\cap_{v\in\VV}\,\C[v][\X]$. In fact, equality holds, i.e., every polytope is equal to the intersection of its vertex cones. We will use  a stronger version of this statement given in the following lemma.  The result, we believe, is folklore, but since we couldn't find a reference, we give the proof for completeness.

\begin{lemma}\label{lem:cone}
For any $\emptyset\neq S\subseteq\VV$, we have $\X=\cap_{v\in S}\,\C[v][\X]$ if and only if every facet of $\X$ contains some $v\in S$.
\end{lemma}

A special case of the above result arises by considering $S=\VV$, which leads to 
\begin{equation}
\X = \bigcap_{v\in\VV}\,\C[v][\X].
\end{equation}

\begin{proof}

Suppose that $\X=\cap_{v\in S}\,\C[v][\X]$. By~\eqref{eq:cone2}, $\C[v][\X] = \{x\colon a_{i}x \ge a_{i}v,\ i\in B^{\prime}(v) \}$, where $B^{\prime}(v)$ is the set of tight inequalities at $v$. This implies $\X = \{x\mid a_{i} x \ge a_{i} v, \ i \in B^{\prime}(v), v\in S\}$. Note that for every facet of a polyhedron there exists some defining inequality of the polyhedron that represents this facet. Hence if $F$ is a facet of $\X$, then $F=\{x\in\X\mid a_{i} x = a_{i} v \}$ for some $i \in B^{\prime}(v), v\in S$. Then it is clear that $v\in F$. 

For the reverse direction, we will need the following.
\begin{claim}\label{clm:facetcorr}
Let $H$ be a supporting hyperplane of $\X$. Then $\X\cap H$ is a facet of $\X$ if and only if  $\C[v][\X]\cap H$ is a facet of $\C[v][\X]$ for every $v\in\VV\cap H$.
\end{claim}
\begin{proof}
We have that $H$ defines a proper face of $\X$, i.e. the dimension of $\X\cap H$ is at least 0 and at most $d-1$. $(\Longleftarrow)$  Since $\C[v][\X]$ is a $d$-dimensional polyhedral cone, any $(d-1)$ of its generators are linearly independent, meaning that $v$ and any $(d-1)$-subset of $\N[v][\X]$ are affinely independent. Suppose $H$ defines a facet of $\C[v][\X]$ for every $v\in\VV\cap H$. Then $H$ contains $v$ and at least $d-1$ vertices in $\N[v][\X]$. Therefore $H$ contains $d$ affinely independent vertices of $\X$, making $\X \cap H$ a facet of $\X$.

$(\Longrightarrow)$ Suppose $\X\cap H$ is a facet of $\X$. The cone $\C[v][\X]$ being $d$-dimensional for every $v$, we need to argue that $H$ defines a $(d-1)$-dimensional face of $\C[v][\X]$ for every $v\in\VV\cap H$. For every $v\in\VV$, $\X\subseteq\C[v][\X]$ tells us that $\C[v][\X]\nsubseteq H$ and that the points in  $\N[v][\X]\setminus H$ are all on one side of $H$. Thus for every $v\in\VV\cap H$, the generators $\{u-v\}_{u\in\N[v][\X]}$ of $\C[v][\X]$ belong to one of the halfspaces defined by $H$. Hence $H$ defines a face of $\C[v][\X]$. 
Due to $\C[v][\X]\nsubseteq H$, the dimension of this face is at most $d-1$.  Since $H$ defines a facet of $\X$, we have $\ext(\X\cap H) = \ext(\X)\cap H$ and so $H$ contains $d$ affinely independent vertices of $\X$. Now $\ext(\X)\cap H \subseteq \C[v][\X]\cap H$ tells us that the dimension of the face $\C[v][\X]\cap H$ is at least $d-1$, thereby implying that $H$ defines a facet of $\C[v][\X]$ for every $v\in\VV\cap H$.
\end{proof}

Now suppose every facet of $\X$ contains some $v\in S$. It suffices to prove that $\cap_{v\in S}\,\C[v][\X]\subseteq\X$ because $\X \subseteq\cap_{v\in S}\,\C[v][\X]$ is obvious from \eqref{eq:cone} and $S\subseteq\VV$. For sake of contradiction, let $\x\in\cap_{v\in S}\,\C[v][\X]\setminus\X$. Then $c\x > c_{0}$ for some facet-defining inequality $cx\le c_{0}$ of $\X$. By assumption, there exists some $\bar{v}\in S$ such that $c\bar{v}=c_{0}$. By Claim~\ref{clm:facetcorr}, we have that $cx\le c_{0}$ is a facet-defining inequality of $\C[\bar{v}][\X]$. But then $c\x > c_{0}$ leads to the contradiction $\x\notin\C[\bar{v}][\X]$. 
\end{proof}

\begin{remark}
\newcommand{\reccone}{\operatorname{rec}}
Lemma~\ref{lem:cone} also holds for pointed $d$-polyhedra. Let $\X$ be a $d$-polyhedron with $\VV\neq\emptyset$ and the recession cone $\reccone(\X)  = \operatorname{cone}\{r^{1},\dots,r^{l} \}$. For $v\in\VV$, let $R_{\X}(v) := \{r^{i}\mid v+r^{i} \text{ is an edge of $\X$}\}$. The tangent cone at each vertex $v$ is \[\C[v][\X] =  v + \operatorname{cone}\{\x-v \}_{\x\in\N[v][\X]} + \operatorname{cone}\{r^{i}\}_{r^{i}\in R_{\X}(v)}. \] Then the above proof naturally extends to give us the same characterization for $\X=\cap_{v\in S}\C[v][\X]$.
\end{remark}

Lemma~\ref{lem:cone} poses an interesting question: for a $(d,kd)$-polytope $\X$, is there a good lower bound (in terms of $d$ and $k$) on how many vertex cones are required to describe $\X$? The answer does not seem obvious even for $k=2$. Even a simpler question does not seem obvious: is there a characterization of $(d,2d)$-polytopes, or $(d,n)$-polytopes,  that are equal to the intersection of two vertex cones, which is the minimal number required for any polytope? For $(3,6)$-polytopes, which are called hexahedra and have seven distinct combinatorial types as enumerated in \citep{numericana}, one can graphically verify that every $(3,6)$-polytope is equal to the intersection of two of its vertex cones. For $d=4$, for the dual of the simplicial $4$-polytope wtih $8$ vertices $P_{1}^{8}$ in \citep[pp. 454]{grunbaum1967enumeration}, it is easy to verify that there does not exist any vertex pair $(u,v)$ such that every facet of $P_{1}^{8^{\circ}}$ contains either $u$ or $v$, meaning that $P_{1}^{8^{\circ}}$ requires at least three vertex cones for its description. For general $d$, the answer to the second question is clearly yes for Dantzig figures, due to Lemma~\ref{lem:cone}.
\begin{corollary}\label{corr:dantcone}
When $\Dant{\X}{u}{v}$, we have $\X = \C[u][\X]\cap\C[v][\X]$.
\end{corollary}
The converse of Corollary~\ref{corr:dantcone} is not true --- not every $(d,2d)$-polytope that is equal to the intersection of two vertex cones is a Dantzig figure; for example in $\real^{3}$, a pyramid with a pentagonal base is not a Dantzig figure since every pair of vertices shares a common facet. Therefore, the Dantzig figure property \myred{is not} necessary for a $(d,2d)$-polytope to be described by two cones.


Corollary~\ref{corr:dantcone} can be used to derive an explicit $\fancyH$-representation for a Dantzig figure (and also for any simple polytope). Since $u$ and $v$ have exactly $d$ neighboring vertices, we may denote  $\N[u][\X] = \{y^{1},\dots,y^{d} \}$ and $\N[v][\X] = \{z^{1},\dots,z^{d}\}$. If we let \[ M_{u} := \left[\begin{array}{cccccc}y^{1} - u & y^{2} - u & \cdots & y^{d}-u \end{array}\right], \qquad M_{v} := \left[\begin{array}{cccccc}z^{1} - v & z^{2} - v & \cdots & z^{d}-v \end{array}\right] \] denote the $d\times d$ matrices defined by neighbors of $u$ and $v$, respectively, then the vertex cones $\C[u][\X]$ and $\C[v][\X]$ are given by $\C[u][\X] = u + \{M_{u}\alpha\colon \alpha \ge \zerovec \}$ and $\C[v][\X] = v + \{M_{v}\alpha\colon \alpha \ge \zerovec \}$. Since the above  matrices are nonsingular, these cones are simplicial and we have $\C[u][\X] = \{x\colon M_{u}^{-1}(x-u)\ge\zerovec \}$ and $\C[v][\X] = \{x\colon M_{v}^{-1}(x-v)\ge\zerovec \}$. This combined with Corollary~\ref{corr:dantcone} yields the following.

\begin{proposition}\label{prop:dantcone}
When $\Dant{\X}{u}{v}$, we have the following minimal inequality representation: \[\X = \{x\mid M_{u}^{-1}(x-u)\ge\zerovec, \; M_{v}^{-1}(x-v)\ge\zerovec\}. \]
\end{proposition}

We will apply this method of deriving the $\fancyH$-representation to our polytopes $\P$ and $\Q$. \myred{We remark that the matrix inverses $M_{u}^{-1}$ and $M_{v}^{-1}$ may lead to highly ill-conditioned coefficients for facet-defining inequalities of a Dantzig figure, as will be the case for $\P$ and $\Q$. This would not make the $\fancyH$-representation of Proposition~\ref{prop:dantcone} suitable for computational implementation. In that case, one would seek an extension of the Dantzig figure,  where, as is customary in literature, an extension of a polytope $\X\subset\real^{d}$ is a polyhedron $Y\subset\real^{d^{\prime}}$ and an affine map $\pi\colon\real^{d^{\prime}}\mapsto\real^{d}$ such that $\X = \pi(Y)$. The size of an extension $(Y,\pi)$ is counted by the number of facet-defining inequalities in $Y$. Corollary~\ref{corr:dantcone} gives us an extension of size $2d$  
and the coefficients of the inequalities of this extension are more well-conditioned than those in Proposition~\ref{prop:dantcone} describing the Dantzig figure in the $x$-space.}

\section{The grlex polytope $\P$}\label{sec:grlex}
\renewcommand{\V}{\ext(\P)}
In this section we will describe the main properties of the polytope $\P$. To simplify the notation, throughout, we will use $\preceq$ to denote the grlex order.
\subsection{$\fancyV$-polytope}
Consider the following integral points:
\begin{subequations}\label{eq:points}
\begin{eqnarray}
\y &:=& (\b-1)\onevec[d] \;=\; (0,0,\ldots,0,\b-1) \\
\five{k} &:=& ( (\bb_{k-1}+1)\onevec[k-1], \theta_{k}-1, \theta_{k+1},\ldots,\theta_{d}) \qquad 3 \le k\le d \\
\six{j}{k} &:=& \myred{( \bb_{k}\onevec[j], 0, \ldots, 0,\theta_{k+1},\ldots,\theta_{d})} \qquad 1 \le j < k \le d
\end{eqnarray}
\myred{where $\bb_{k}$ is given by \eqref{b-bk}.}
\end{subequations}
By construction, we have
\begin{observation}
$\five{k}\in\pink$ for all $k$, $\six{j}{k}\in\pink$ for all $j,k$, $\theta\in\pink$, $\zerovec,\y\notin\pink$.
\end{observation}
Since $\theta\ge\onevec$, we have $\y,\five{k},\six{j}{k}\ge\zerovec$. Also, every $\five{k}$ and $\six{j}{k}$ is $\lex$-less than $\theta$. Thus $\y,\five{k},\six{j}{k}\in\P$. Observe that $\theta\ge\onevec$ implies that $\six{j}{k_{1}}$ and $\five{k_{2}}$ coincide if and only if $k_{1}=k_{2}$, $j=k_{1}-1$, and $\theta_{k_{1}}=1$.

Our first result shows that the points defined in \eqref{eq:points}, \myred{along with} $\zerovec$ and $\theta$, provide a vertex characterization of $\P$.

\begin{proposition}\label{prop:V}
The vertices of $\P$ are
\[
\V = \{\zerovec,\theta,\y\} \,\bigcup\, \{ \five{k} \colon 3 \leq k \leq \n \} \,\bigcup\, \{ \six{j}{k}\colon 1 \leq j < k \leq \n \}.
\]
\end{proposition}
\begin{proof}
It is clear from the definition of $\P$ that $\zerovec$ and $\y$ cannot be written as a nontrivial convex combination of integral points in $\P$. 
Suppose $\theta = \sum_{i=1}^s \lambda_i x^i $ is a nontrivial convex combination of some $x^i \in \P\cap\ints^{d}$. Since $\theta\in\P\cap\pink$ and $\pink$ defines a face of $\P$, we have $x^{i}\in\P\cap\pink$ for all $i$. Let \[m = \max\{j \colon x^i_j \neq \theta_j \text{ for some } i \}.\]
Then $x^{i}\lex\theta$ implies $x^i_m \leq \theta_m$, leading to the contradiction $\sum_{i=1}^s \lambda_i x^i_m < \theta_m$. Next, suppose $\five{k} = \sum_{i=1}^s \lambda_i x^i$ is a nontrivial convex combination of some $x^i \in \P\cap\ints^{d}$. Since $\five{k} \in \pink$, $x^i \in \pink$ as well. By the same reasoning as for $\theta$ we get $x^i_j = \theta_j$ for all $i$ and $j>k$. Also, $x^i_j = 0$ for all $i$ and $j<k-1$. Now, if $x^i_k = \theta_k$ for some $i$ then $x^i_{k-1} = \bb_{k-1}$ which contradicts $x^i_{k-1} \grlex \theta$ because $\theta \geq \onevec$ and $k \geq 3$. So, the only possibility is $x^i = \five{k}$ for all $i$. Similarly, one can conclude that all $\six{j}{k}$ have to be vertices of $\P$ as well.

{
\renewcommand{\ell}{k}
\renewcommand{\x}{\bar{x}}
\renewcommand{\chi}{\sigma}  
Now we argue that if $v \in \V\setminus\{\zerovec,\y,\theta\}$, then $v$ must be equal to some $\five{k}$ or $\six{j}{k}$. Equation~\eqref{eq:Punion} gives us \[\V \subseteq \ext \{x\in\real^{d}\colon \sum_{i}x_{i} \le \b-1\} \cup \ext(\P\cap\pink).\] The vertices of the simplex $ \{x\in\real^{d}\colon \sum_{i}x_{i} \le \b-1\}$ are $\zerovec$ and $(\b-1)\onevec[i]$ for $i=1,\ldots, d$. The point $(\b-1)\onevec[i] $ is a convex combination of $\zerovec$ and $\b\onevec[i]$ and for $i \le d-1$, $\theta_{d}\ge 1$ implies $\b\onevec[i]\lex\theta$, and hence $\b\onevec[i]\in\P\cap\pink$. Therefore $(\b-1)\onevec[i]\notin\V$ for $i \le d-1$ and we have 
\[
\V\subseteq\{\zerovec,\y\}\cup\ext(\P\cap\pink).
\]
Since $\pink$ defines a face of $\P$, we have $\ext(\P\cap\pink)=\ext(\P)\cap\pink = \V\cap\pink$ and then $\zerovec,\y\in\V$ leads to the equality
\[
\V=\{\zerovec,\y\}\cup\ext(\P\cap\pink).
\]
We argued in the first paragraph that $\theta\in\V\cap\pink$ and because $\V\cap\pink=\ext(\P\cap\pink)$, we have $\theta\in\ext(\P\cap\pink)$. Let $\x\neq\theta$ be an arbitrary  vertex of $\P\cap\pink$ and define $\ell := \max\{i\colon \x_{i}\neq\theta_{i} \}$. Since $\V\subseteq\ints^{d}$, we have $\x\in\ints^{d}$ and $\x_{\ell} \in \{0,1,\ldots, \theta_{\ell}-1\}$. Suppose $1\le \x_{\ell}\le\theta_{\ell}-2$. Then $\x\in\pink$ implies there exists a $i < \ell$ such that $1 \le \x_{i} \le \b-1$. For $x^{\prime},x^{\prime\prime}\in\pink\cap\ints^{d}$ defined as $x^{\prime} = \x + \onevec[i] - \onevec[k]$ and $x^{\prime\prime} = \x - \onevec[i] + \onevec[k]$, note that $1\le\x_{\ell}\le\theta_{\ell}-2$ implies that $\zerovec\le x^{\prime},x^{\prime\prime}\lex\theta$. Hence $x^{\prime},x^{\prime\prime}\in\P\cap\pink$ and since $\x=(x^{\prime}+x^{\prime\prime})/2$, we have a contradiction to $\x\in\ext(\P\cap\pink)$, thereby implying that $\x_{\ell}\in\{0,\theta_{\ell}-1\}$. The assumption $\theta\ge\onevec$ allows us to make similar arguments when $\x_{\ell}\in\{0,\theta_{\ell}-1\}$ and $\x_{i},\x_{j} \ge 1$ for distinct $i,j \le \ell-1$. Thus if $\x_{\ell}=0$, then $\x\in\ext(\P\cap\pink)$ only if $\x=\six{j}{\ell}$ for some $j\le \ell-1$. Finally let $\x_{\ell}=\theta_{\ell}-1$, $\x_{i} = \bb_{k-1}+1$ for some $i \le k-2$ and $\x_{t}=0$ for $t\neq i,k$. In this case, $\x = \frac{\theta_{\ell}-1}{\theta_{\ell}}\chi^{1} + \frac{1}{\theta_{\ell}}\chi^{2}$, where
\myred{\[
\chi^{1} := (\bb_{\ell-1}\onevec[i],\theta_{\ell},\theta_{\ell+1},\ldots,\theta_{d}), \quad \chi^{2} := (\bb_{\ell}\onevec[i],0,\theta_{\ell+1},\theta_{\ell+2},\ldots,\theta_{d}),
\] }
and $\chi^{1},\chi^{2}\in\P\cap\pink$ due to $\theta\ge\onevec,\ell\ge 3,i\le\ell-2$. Hence $\x$ must be equal to $\five{\ell}$ when $\x_{\ell}=\theta_{\ell}-1$.
}
\end{proof}

\begin{observation}\label{obs:coord}
The $\n$ coordinate planes are facets of $\P$, which we call trivial facets. 
\end{observation}
\begin{proof}
We know $\zerovec\in\P$. The assumption $\theta\ge\onevec$ implies that $\onevec[i]\grlex\theta$ for $1\le i\le\n$.
\end{proof}

The remaining facets are defined by supporting hyperplanes that have a monotone coefficient property, which we prove next. 

\begin{lemma}\label{l:monotone}
Suppose $\P \subseteq \{x\colon cx \leq c_0\}$ and let $F =\P \cap  \{x\colon cx = c_0\}$ be a face of $\P$. If $F$ is not contained in $x_{i+1} = 0$ for some $i \in \{1, \ldots, d-1\}$ then $ c_{i+1} \geq \max\{c_i, 0\}$. Consequently, if $F$ is a nontrivial facet then  $0 \leq c_i \leq c_{i+1}$ for all $1 \leq i \leq d-1$.
\end{lemma}
\begin{proof} Let $\bar{x}$ be a vertex on $F$ with $\bar{x}_{i+1} \ge 1$. Consider first $x' = \bar{x} -\onevec[i+1]$.  Since $\zerovec \leq x^{\prime} \grlex \bar{x}$, we have $c x^{\prime} \leq c_0 = c\bar{x}$, which implies $c_{i+1} \geq 0$.  Similarly, the point $x^{\prime\prime} = \bar{x} + \onevec[i]-\onevec[i+1]$ also has the property $\zerovec \leq x^{\prime\prime} \grlex \bar{x}$. Therefore, we have $c x^{\prime\prime} \leq c_0 = c\bar{x}$, which yields $c_i \leq c_{i+1}$. \end{proof}

We also note that $\pink$ defines a facet of $\P$.

\begin{observation}\label{obs:pinkfacet}
$\sum_{i=1}^{d}x_{i} \le \b$ is a facet-defining inequality for $\P$.
\end{observation}
\begin{proof}
The face $\P\cap\pink$ contains $\theta$ and the $\n-1$ coordinate vectors $\six{1}{d},\six{2}{d},\ldots,\six{d-1}{d}$, and these vertices are affinely independent because of $\theta_{d}\ge 1$.
\end{proof}

In proving our first main result Theorem~\ref{thm:dantzig} and deriving the $\fancyH$-representation of $\P$, the adjacancies of $\zerovec$ and $\theta$ will be useful. 

\begin{proposition}\label{prop:neighbors} 
The neighbors of $\theta$  and $\zerovec$ are 
\begin{align*}
\N &= \{\y, \six{1}{2} \} \cup \{\five{k} : 3 \leq k \leq d\} \\
\N[\zerovec] &= \{\y \} \cup \{\six{j}{d} \colon 1 \le j \le d-1\}. 
\end{align*}
\end{proposition}
\begin{proof} Since $\theta$ has at least $d$ neighbors, it suffices  to show that the other vertices are not neighbors of $\theta$. Now, suppose $\P \subseteq \{x\colon cx \leq c_0\}$ and $F = \P \cap \{x \colon cx = c_0\}$  is an edge through $\theta$ and $\six{j}{k}$ for some $k \geq 3$ and $j < k$.  Lemma~\ref{l:monotone} implies
\[ (c_{k-1}-c_j) (\theta_1 + \cdots + \theta_{k-1} + 1) + (c_k -c_j)(\theta_k -1) \geq 0\]
which, in turn, implies $c \five{k} \geq c \six{j}{k} =c_0$. Therefore $\five{k} \in F$, which contradicts the assumption that $F$ is an edge. 

\myred{Consider any vertex $v$ from the list $\{\y,\b\onevec[1],\ldots,\b\onevec[d-1]\}$. Note that $\b\onevec[j] = \six{j}{d}$. Since the coordinate planes define trivial facets of $\P$, $v$ has at least $d-1$ common facets with $\zerovec$. Then the fact that $v$ is a coordinate vector implies that there cannot be another vertex on the affine span of $\zerovec$ and $v$. Therefore, $v\in\N[\zerovec]$. The conic hull of all these neighbors is $\real^{d}_{+}$}. Hence $\C[\zerovec]$, the vertex cone at $\zerovec$, contains $\real^{d}_{+}$ but since $\P\subseteq\real^{d}_{+}$, we in fact have $\C[\zerovec]=\real^{d}_{+}$. Therefore there does not exist another vertex of $\P$ which is a neighbor of $\zerovec$.
\end{proof}



\myred{\begin{theorem}\label{thm:dantzig}
$\Dant{\P}{\zerovec}{\theta}$ and $(\zerovec,\theta)$ is the only antipodal vertex pair of $\P$. 
\end{theorem}
}
\begin{proof} 
Proposition~\ref{prop:neighbors} gives us that each of $\zerovec$ and $\theta$ has exactly $d$ neighboring vertices. For any vertex $v$ of a $d$-polytope, every facet containing $v$ also contains at least $d-1$ neighbors of $v$. Hence each of $\zerovec$ and $\theta$ lies on  exactly $d$ facets. The nonnegativity of the coefficients of the supporting hyperplanes from Lemma~\ref{l:monotone} and the assumption $\theta\ge\onevec$ imply that there is no nontrivial face containing both $\zerovec$ and $\theta$. 

Now we need to prove that every facet of $\P$ contains either $\zerovec$ or $\theta$. 
Suppose $F$ is a facet of $\P$ given by $c x \le c_0$.  If $F$ doesn't contain $\theta$ nor any of the vertices $\six{1}{k}$, then it is contained in the subspace $x_1=0$ and hence be equal to the facet defined by $x_{1}\ge 0$ and therefore contain $\zerovec$. So, suppose $F$ contains a vertex $\six{1}{k}$ for some $k$ and suppose $c \zerovec <c_0$ and $c \theta < c_0$. Then $\six{1}{k} \in F$ implies 
\[c \theta < c_0 = c_1 \sum_{i=1}^k \theta_i + \sum_{i=k+1}^d c_i \theta_i \]
and using Lemma~\ref{l:monotone} we get
\[0 \leq \sum_{i=2}^k (c_i -c_1) \theta_i <0\]
which is a contradiction.

The vertex $\y$ is a neighbor of both $\zerovec$ and $\theta$ and the vertices $\theta$, $\five{k}$ and $\six{j}{k}$ all belong to the facet defined by $\pink$. Hence any other antipodal pair of vertices must be of the form $(\zerovec,v)$ or $(\y,v)$, where $v=\five{k}$ for some $3 \le k \le d$, or $v=\six{j}{k}$ for some $1\le j < k \le d$. However each of these pairs has a common facet defined by some coordinate plane: each of the two pairs $(\zerovec,\five{k})$ and $(\y,\five{k})$ shares the facet $x_{1}=0$; $(\zerovec,\six{j}{k})$ share the facet $x_{k}=0$; $(\y,\six{j}{k})$ share the facet $x_{k}=0$ if $k < d$ and otherwise, $x_{i}=0$ for some $i\le d-1, i\neq j$.
\end{proof}

\subsection{$\fancyH$-polytope}


Proposition~\ref{prop:neighbors} tells us that the vertex cone at $\zerovec$ is $\C[\zerovec] = \real^{d}_{+}$. Then Proposition~\ref{prop:dantcone} gives us $\P = \{x\ge\zerovec\mid M^{-1}(x-\theta)\ge\zerovec\}$, where
\begin{eqnarray*}
M &:=& 
\left[\begin{array}{cccccc}\six{1}{2} - \theta & \five{3} - \theta & \five{4}-\theta & \cdots & \five{d}-\theta & \y-\theta\end{array}\right] \medskip\\
& = & 
\left[\begin{array}{ccccccc}
\theta_2 & -\theta_1 & -\theta_1 & -\theta_1 & \cdots & -\theta_{1} & -\theta_1 \\
-\theta_2 & \bb_{1} + 1 & -\theta_2 & -\theta_2 & \cdots & -\theta_{2} & -\theta_2 \\
0 & -1 & \bb_{2} + 1 & -\theta_3 & \cdots & -\theta_{3} & -\theta_3 \\
\vdots & 0 & -1 & \bb_{3}+1 & \cdots & -\theta_{4} & -\theta_4 \\
\vdots & \vdots & 0 & -1 & \ddots & \vdots & \vdots \\
\vdots & \vdots & \vdots & \vdots & \ddots & \bb_{d-2}+1 & \myred{-\theta_{d-1}} \\
0 & \cdots & \cdots & \cdots & \cdots & -1 &  \bb_{d-1}-1 
\end{array}\right]
\end{eqnarray*}

To describe the inverse of $M$, denote 
\[ p_{i}^{j} := \begin{cases} \bb_{i}\prod_{k=i+1}^{j} (\bb_{k}+1) & j >i \medskip\\
                                         \bb_{i} & j=i \medskip \\
                                         1 & j < i.
                  \end{cases}
 \]

\begin{proposition}\label{prop:N}
$\P = \left\{x \ge \zerovec\mid Nx \ge N\theta \right\}$ where $N=M^{-1}$ with
\[
\begin{split}
N_{\n,i} = -1, \ 1 \le i \le \n, \quad N_{i,\n} = -p_{i}^{\n-1}, \ 2 \le i \le \n, \quad N_{1,\n} = \frac{-p_{1}^{\n-1}}{\theta_{2}}, \\
 N_{i,j} = \begin{cases}\displaystyle
N_{i,j+1} + \frac{p_{1}^{j-1}}{\theta_{2}} & i=1,1\le j \le \n-1 \medskip\\
\displaystyle N_{i,j+1} + p_{i}^{j-1} & 2 \le i \le j\le\n-1 \medskip\\
N_{i,j+1} & 1\le j < i \le \n-1.
\end{cases}
\end{split}
\]

\end{proposition}


\begin{proof}

We need to consider several cases when computing $N_{i\cdot} M_{\cdot j}$. First note that 
\[N_{1\cdot} M_{\cdot 1} = \theta_{2}(N_{1,1} - N_{1,2}) = \theta_{2}/\theta_{2} = 1 \]
and for $i\ge 2$, 
\begin{align*}
N_{i\cdot} M_{\cdot i} &= -\sum_{l=1}^{i-1} N_{i,l}\theta_{i}  + (\bb_{i-1}+1)N_{i,i} - N_{i,i+1} \\
&= - N_{i,i}\sum_{l=1}^{i-1} \theta_{l}  + (\bb_{i-1}+1)N_{i,i} - N_{i,i+1} \\
&= N_{i,i} - N_{i,i+1}\\
&= 1.
\end{align*}

Consider now the case $1 \le j < i \le \n$. It is readily seen that 
\[N_{\n\cdot}M_{\cdot j} = -\sum_{l=1}^{\n}M_{l,j} = 0\] and for $i\le\n-1$, 
\begin{align*}
N_{i\cdot} M_{\cdot j} &= -N_{i,i}\sum_{l=1}^{j-1} \theta_{l} + (\bb_{j-1}+1)N_{i,j} - N_{i,j+1} \\
&= -N_{i,i}\sum_{l=1}^{j-1} \theta_{l} + (\bb_{j-1}+1)N_{i,i} - N_{i,i} \\
&= 0.
\end{align*}

For $ 1 \leq i < j \leq \n-1$,
\begin{align*}
N_{i \cdot} M_{\cdot j}  - N_{i \cdot} M_{\dot j-1}&=N_{i \cdot} (M_{\dot j}  - M_{\cdot j-1}) \\
&= N_{i,j-1}(\theta_{j-1}- \bb_{j-2}-1) + N_{j}(\bb_{j-1}+2) + N_{i, j+1}(-1)\\
&= -N_{i, j-1}(b_{j-1}+1) + N_{i,j}(b_{j-1}+2) - N_{i, j+1}\\
&= (b_{j-1}+1) (N_{i,j} - N_{i,j-1}) + (N_{i,j} - N_{i,j+1})\\
&= \begin{cases}
(b_{j-1}+1) (\frac{-p_{i}^{j-2}}{\theta_{2}}) + \frac{p_{i}^{j-1}}{\theta_{2}}, \;\; & 1=i < j \leq \n -1\\
(b_{j-1}+1) (-p_{i}^{j-2}) + p_{i}^{j-1}, \;\; & 1< i < j \leq \n -1\\
\end{cases}\\
& = \begin{cases}
0, \;\; & i+1 = j \leq d-1\\
-1, \;\; & i+1 < j \leq d-1.\\
\end{cases}
\end{align*}
Therefore, $N_{i \cdot} M_{\cdot j} =0$ for $1 \leq i < j \leq \n -1$.

\begin{align*}
N_{i \cdot} M_{\cdot \n}  - N_{i \cdot} M_{\cdot \n-1} &= N_{i \cdot} (M_{\cdot \n}  - M_{\cdot \n-1})\\
& = N_{i, \n-1}(-\bb_{\n-2} - 1 -\theta_{\n-1}) + N_{i, \n}(1+\bb_{\n-1} -1)\\
& = \begin{cases}
\left(N_{i,\n} + \frac{p_{i}^{\n-2}}{\theta_{2}}\right)(\bb_{\n-1} +1) - \bb_{\n-1}N_{i,\n}, \;\; & i=1\\
(N_{i,\n} + p_{i}^{\n-2})(\bb_{\n-1} +1) - \bb_{\n-1}N_{i,\n}, \;\; & 1 < i \leq \n-1
\end{cases}\\
& = \begin{cases}
N_{i,\n} + \frac{p_{i}^{\n-2}}{\theta_{2}}(\bb_{\n-1} +1), \;\; & i=1\\
N_{i,\n} + p_{i}^{\n-2}(\bb_{\n-1} +1), \;\; & 1 < i \leq \n -1
\end{cases}\\
& = \begin{cases}
0, \;\; & 1 \leq i \leq d-2\\
-1, \;\; & i =\n-1.
\end{cases}
\end{align*}
Therefore,
$N_{i \cdot} M_{\cdot \n} =0$ for $1 \leq i \leq  \n -1$.
\end{proof}

\subsection{Graph of the polytope}\label{sec:G}
\newcommand{\coord}[1]{H_{#1}}
\newcommand{\Hcal}[1][\P]{\mathcal{H}^{\theta}_{#1}}
\newcommand{\Ccal}[1][]{
\ifthenelse{\isempty{#1}}
{\mathcal{H}^{\zerovec}}
{\mathcal{H}^{\zerovec,#1}}
}
\renewcommand{\phi}{\psi_{\P}}

Let $\G{\P}$ denote the graph of $\P$. Based on Proposition~\ref{prop:V}, it is clear that \myred{if $\theta > 1$}, $\G{\P}$ has $\frac{\n^{2}+\n+2}{2}$ vertices. We next characterize the facet-vertex incidence for $\P$. This leads us to finding all the edges of $\G{\P}$ since $\P$ is a $\n$-polytope and so for any $v,v^{\prime}\in\V$, $(v,v^{\prime})$ is an edge in $\G{\P}$ if and only if there are at least $\n-1$ facets incident to both $v$ and $v^{\prime}$ \myred{and all these common facets are not incident to another vertex $v^{\prime\prime}$}. As can be seen from Corollary~\ref{corr:edges}, $\G{\P}$ depends only on which entries of  $\theta$ are 1, and for a fixed $\n$, these graphs are isomorphic for all $\theta > \onevec$ (Figure~\ref{fig:graphG}). We then derive some of the basic properties of $\G{\P}$ such as radius, diameter, coloring number, etc. We also show that when $\theta > \onevec$,   the edge expansion of this graph is equal to one.

\begin{figure}[ht]
\centering
\resizebox{11cm}{!}{
\begin{tikzpicture}[line width=1pt]
\sffamily
\node (a1) {$v^{1,2}$};
\node[below=of a1] (a2) {};

\node[right=1cm of a1] (aux1) {$v^{1,3}$};
\node[below= 0.1cm of aux1] (b1) {$v^{2,3}$};

\node[right=3cm of aux1] (aux2) {$v^{1,k-2}$};
\node[below= 0.1cm of aux2] (c1) {$v^{2,k-2}$};
\node[below= 0.1cm of c1] (c2) {\vdots};
\node[below= 0.1cm of c2] (c3) {$v^{k-3,k-2}$};

\node[right=1cm of aux1] (b115) {};
\node[left=1cm of aux2] (b116) {};
\node[right= 1cm of b1] (b15) {$\cdots$};

\node[right= 1cm of aux2] (aux26) {};
\node[right= 1cm of c1] (c16) {};

\draw[-,myblue] (aux2) -- (aux26);
\draw[-,myblue] (c1) -- (c16);

\node[right=1cm of aux1] (b115) {};
\node[left=1cm of aux2] (b116) {};
\node[right= 1cm of b1] (b15) {$\cdots$};
\node[left= 1cm of c1] (b16) {};

\node[right= 1.5cm of c2] (c25) {$\cdots$};

\node[right=3cm of aux2] (aux3) {$v^{1,k}$};
\node[below= 0.1cm of aux3] (d1) {$v^{2,k}$};
\node[below= 0.1cm of d1] (d2) {};
\node[below= 0.1cm of d2] (d3) {\vdots};
\node[below= 0.1cm of d3] (d4) {};
\node[below= 0.1cm of d4] (d5) {$v^{k-1,k}$};

\node[right=1cm of aux3] (aux36) {};
\node[left=1cm of aux3] (aux35) {};
\node[right=1cm of d1] (d16) {};
\node[left=1cm of d1] (d15) {};
\node[right=1cm of d5] (d56) {};

\draw[-,myblue] (aux3) -- (aux36);
\draw[-,myblue] (aux3) -- (aux35);
\draw[-,myblue] (d1) -- (d16);
\draw[-,myblue] (d1) -- (d15);
\draw[-,myblue] (d5) -- (d56);

\node[right= 1.2cm of d3] (d35) {$\cdots$};

\node[right=3cm of aux3] (aux4) {$v^{1,d}$};
\node[below= 0.1cm of aux4] (e1) {$v^{2,d}$};
\node[below= 0.1cm of e1] (e2) {};
\node[below= 0.1cm of e2] (e3) {\vdots};
\node[below= 0.1cm of e3] (e4) {};
\node[below= 0.1cm of e4] (e5) {};
\node[below= 0.1cm of e5] (e6) {$v^{d-1,d}$};

\node[left=1cm of aux4] (aux45) {};
\node[left=1cm of e1] (e15) {};

\draw[-,myblue] (aux4) -- (aux45);
\draw[-,myblue] (e1) -- (e15);

\node[right=2cm of e3] (z) {$\boldsymbol{0}$};

\node[above=3cm of aux2] (w) {$w$};

\draw [-,bend left=60,myblue] (w) to (z);

\node[below=6cm of a1](t){$\theta$};
\node[right=1cm of t](u3){$u^{3}$};
\node[right=1cm of u3](u35){$\cdots$};
\node[right=3cm of u3](uk2){$u^{k-2}$};
\node[right=1cm of uk2](uk5){$\cdots$};
\node[right=3cm of uk2](uk){$u^{k}$};
\node[right=1cm of uk](ud5){$\cdots$};
\node[right=3cm of uk](ud){$u^{d}$};

\node[shape=ellipse,draw=myblue,minimum size=1cm,fit={(a1) (a1)}] (el1){};
\node[shape=ellipse,draw=myblue,minimum size=1cm,fit={(aux1) (b1)}] (el2){};
\node[shape=ellipse,draw=myblue,minimum size=1cm,fit={(aux2) (c3)}] (el3){};
\node[shape=ellipse,draw=myblue,minimum size=1cm,fit={(aux3) (d5)}](el4) {};
\node[shape=ellipse,draw=myblue,minimum size=1cm,fit={(aux4) (e6)}] (el5){};
\node[shape=ellipse,draw=myblue,minimum size=2.5cm,fit={(t) (ud)}](el6) {};

\draw[dashed, myblue, ultra thick] (el1) -- (w);
\draw[dashed, myblue, ultra thick] (el2) -- (w);
\draw[dashed, myblue, ultra thick] (el3) -- (w);
\draw[dashed, myblue, ultra thick] (el4.90) -- (w);

\draw[dashed, myblue, ultra thick] (el3.300) -- (uk);
\draw[dashed, myblue, ultra thick, bend right=10] (el3.300) to (ud);

\draw[dashed, myblue, bend right=30, ultra thick] (el1.270) to (uk2);
\draw[dashed, myblue, bend right=30, ultra thick] (el1.270) to (uk);
\draw[dashed, bend right=20, myblue, ultra thick] (el1.270) to (ud);

\draw[dashed, myblue, bend right=30, ultra thick] (el2.300) to (uk2);
\draw[dashed, myblue, bend right=30, ultra thick] (el2.300) to (uk);
\draw[dashed, bend right=10, myblue, ultra thick] (el2.300) to (ud);

\draw[dashed, myblue, ultra thick] (el3.300) -- (uk);
\draw[dashed, myblue, ultra thick, bend right=10] (el3.300) to (ud);

\draw[dashed, myblue, ultra thick] (el4.280) -- (ud);

\draw [dashed,bend right=60,myblue, ultra thick] (w) to (el6.180);

\draw[-,myblue] (a1) -- (aux1.175);
\draw[-,myblue] (z) -- (aux4);
\draw[-,myblue] (z) -- (e1);
\draw[-,myblue] (z) -- (e3);
\draw[-,myblue] (z) -- (e6);

\draw[-,myblue] (t) -- (a1);
\draw[-,myblue] (u3) -- (b1);
\draw[-,myblue] (uk2) -- (c3);
\draw[-,myblue] (uk) -- (d5);
\draw[-,myblue] (ud) -- (e6);

\draw[-,myblue] (aux1) -- (b115);
\draw[-,myblue] (b116) -- (aux2);
\draw[-,myblue] (b16) -- (c1);
\end{tikzpicture}
}
\caption{The graph $\G\P$ for a vertex $\theta > \onevec$ in $\mathbb{R}^{d}$. The circled vertices form cliques. The dashed edges represent connections between a vertex and a clique.}
\label{fig:graphG}
\end{figure}
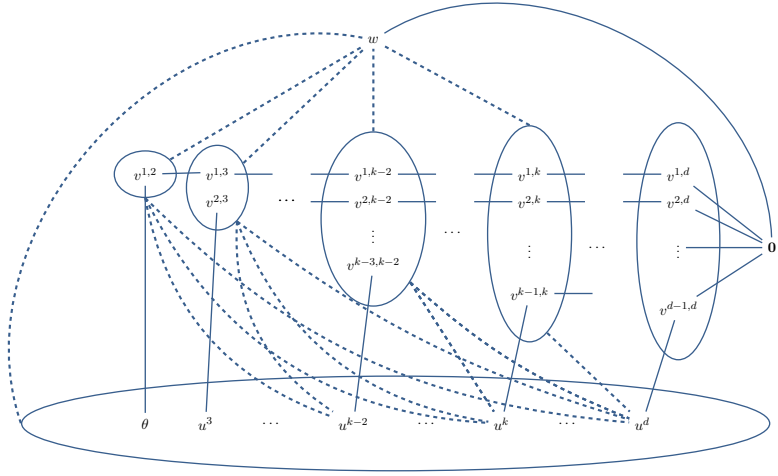

Since $\Dant{\P}{\zerovec}{\theta}$ as per Theorem~\ref{thm:dantzig}, the vertex cones $\C[\zerovec]$ and $\C[\theta]$ are simplicial. Thus the $\n$ facet-defining hyperplanes incident to $\theta$ (resp. $\zerovec$) are in a one-to-one correspondence with the $\n$ neighbors of $\theta$ (resp. $\zerovec$). Furthermore, we know that $\zerovec$ and $\theta$ do not belong to a common facet. 
Hence we denote the facet-defining hyperplanes incident to $\theta$ as: 
\begin{align*}
 H_{\y} &= \{x\colon N_{\n\cdot}(x-\theta) = 0\} = \pink,\quad H_{\six{1}{2}} = \{x\colon N_{1\cdot}(x-\theta)= 0\}, \\
H_{\five{k}} &= \{x\colon N_{(k-1)\cdot}(x-\theta)= 0\} \quad 3 \le k \le \n, 
\end{align*}
where $N$ is the inverse of $M$ from Proposition~\ref{prop:N} and $H_{v}$, for $v\in\N$,  signifies the only facet-defining hyperplane that contains $\theta$ but not $v$. Let $\Hcal$ denote the collection of these hyperplanes, i.e., 
\begin{subequations}
\begin{equation}\label{eq:Htheta}
\Hcal := \{H_{\six{1}{2}}, H_{\five{3}},\ldots,H_{\five{\n}},H_{\y} \}.
\end{equation}
The $\n$ facet-defining hyperplanes incident to $\zerovec$ are the coordinate planes $\coord{i} := \{x\colon x_{i}=0\}$ for $1\le i \le \n$, and we denote this collection by 
\begin{equation}\label{eq:Hzero}
\Ccal:= \{H_{1},H_{2},\ldots,H_{\n} \}, \quad \Ccal[i] := \{H_{1},H_{2},\ldots,H_{i} \} \quad 1\le i\le\n.
\end{equation}
\end{subequations}

The vertex-facet incidence for $\P$ is stated in the following result. For $v\in\V$, let $\phi(v)$ denote the subset of facet-defining hyperplanes of $\P$ that contain $v$.

\begin{proposition}\label{prop:vfacet}
We have
\begin{align*}
\phi(\zerovec) &= \Ccal,\quad \phi(\theta) = \Hcal,\\
\phi(\y) &= (\Hcal\setminus\{H_{\y}\}) \cup \Ccal[\n-1],\\
\phi(\five{k}) &= (\Hcal\setminus\{H_{\five{k}}\}) \cup \Ccal[k-2] \cup \begin{cases}
\coord{k} & \text{ if $\theta_{k}=1$ } \\
\emptyset & \text{ if $\theta_{k}\ge 2$}
\end{cases}, \qquad 3 \le k \le \n\\
\phi(\six{j}{k}) &= (\Hcal\setminus\{H_{\six{1}{2}},H_{\five{3}},\ldots,H_{\five{k}}\}) 
\cup (\Ccal[k]\setminus\{\coord{j} \}),
\quad 1 \le j < k \le \n, (j,k)\neq(2,3)\\
\phi(\six{2}{3}) &= \{H_{1}, H_{3}\}\cup \begin{cases}
(\Hcal\setminus\{H_{\five{3}}\})   & \text{ if } \theta_{3}=1\medskip\\
(\Hcal\setminus\{H_{\six{1}{2}},H_{\five{3}}\})   & \text{ if } \theta_{3}\ge 2.
\end{cases}
\end{align*}
\end{proposition}
\begin{proof}
The expressions for $\phi(\zerovec)$ and $\phi(\theta)$ are obvious. For the other vertices, because it is trivial to check containment in a coordinate plane using our assumption $\theta\ge\onevec$, we only argue the incidence of the elements of $\Hcal$. Since $\y\in\N$, we have $\phi(\y)\supset\Hcal\setminus\{H_{\y}\}$ by construction of $\Hcal$. The value of $\phi(\five{k})$ follows by a similar reasoning. Now fix $k\ge 2$ and $1 \le j \le k-1$. 
Define $\xi := M^{-1}(\six{j}{k}-\theta)$. Proposition~\ref{prop:N} gives us $\xi \ge \zerovec$.  The rows of $M^{-1}$ correspond to the hyperplanes $H_{\six{1}{2}}, H_{\five{3}}, \dots, H_{\five{\n}}, H_{\y}$, respectively. Then to prove the claimed expression for $\phi(\six{j}{k})$, we need to show that $\xi_{k} = \xi_{k+1}=\cdots=\xi_{\n}=0$ and $\xi_{i} > 0$ for $1\le i \le k-1$, except that $\xi_{1}=0$ when $j=2,k=3,\theta_{3}=1$. 

The last row in $M^{-1}$, denoted by $M^{-1}_{\n\cdot}$, is a vector of $-1$'s, meaning that the facet-defining inequality $M^{-1}_{\n\cdot}x \ge M^{-1}_{\n\cdot}\theta$ corresponds  to the hyperplane $\pink$, which we know contains $\six{j}{k}$. Thus 
\begin{subequations}
\begin{equation}\label{eq:xidzero}
\xi_{\n}=M^{-1}_{\n\cdot}(\six{j}{k}-\theta) = 0.
\end{equation}
Now consider the linear system $M\xi = \six{j}{k}-\theta$. The upper Hessenberg structure of $M$ gives us the following recursion:
\begin{align}
\xi_{i-1} &= (\bb_{i-1}+1)\xi_{i} - \theta_{i}\sum_{t=i+1}^{\n}\xi_{t} - \six{j}{k}_{i}+\theta_{i}, \qquad 3 \le i \le \n \label{eq:xirecurse1}\\
\xi_{1} &= \frac{\bb_{1}+1}{\theta_{2}}\xi_{2} - \sum_{t=3}^{\n}\xi_{t} - \frac{\six{j}{k}_{2}}{\theta_{2}}+1 \label{eq:xi1} = \frac{\theta_{1}}{\theta_{2}}\sum_{t=2}^{\n}\xi_{t} + \frac{\six{j}{k}_{1}-\theta_{1}}{\theta_{2}}.
\end{align}
\end{subequations}
Note that
\begin{equation}\label{eq:vthetadiff}
\six{j}{k}_{i} - \theta_{i} = \begin{cases}
0 & k+1\le i \le \n \\
\bb_{k}-\theta_{j}  & i=j\\
-\theta_{i} & \text{otherwise}.
\end{cases}
\end{equation}
First let us apply \eqref{eq:xirecurse1} with $i=\n$. Invoking $\xi_{\n}=0$ from \eqref{eq:xidzero} gives us $\xi_{\n-1}= \theta_{\n}-\six{j}{k}_{\n}$, which is equal to $\theta_{\n}\ge 1$ if $k=\n$, otherwise it is zero. Equation~\eqref{eq:vthetadiff} and a backward induction on $i$ in \eqref{eq:xirecurse1}  then lead us to 
\begin{equation*}
\xi_{k} = \xi_{k+1}=\cdots=\xi_{\n}=0.
\end{equation*}
The expressions for the remaining $\xi$'s can be obtained from \eqref{eq:xirecurse1} and \eqref{eq:xi1} as
\begin{subequations}
\begin{align}
\xi_{i-1} &= (\bb_{i-1}+1)\xi_{i} - \theta_{i}\sum_{t=i+1}^{k-1}\xi_{t} - \six{j}{k}_{i}+\theta_{i}, \qquad 3 \le i \le k \label{eq:xirecurse2} \\
\xi_{1} &= \frac{\bb_{1}+1}{\theta_{2}}\xi_{2} - \sum_{t=3}^{k-1}\xi_{t} - \frac{\six{j}{k}_{2}}{\theta_{2}}+1 = \frac{\theta_{1}}{\theta_{2}}\sum_{t=2}^{k-1}\xi_{t} + \frac{\six{j}{k}_{1}-\theta_{1}}{\theta_{2}}. \label{eq:xi1recurse}
\end{align}
\end{subequations}
For $\xi_{2},\dots,,\xi_{k-1}$, we claim the following:
\begin{subequations}\label{eq:xiformula}
\begin{align}
\xi_{i} &= \bb_{i}\left(\sum_{t=i+1}^{k-1}\xi_{t} - 1\right) + \sum_{t=i+2}^{j-1}\theta_{t} \qquad  i=2,\ldots,j-1 \label{eq:xicase1}\\
\xi_{i} &= \bb_{i}\sum_{t=i+1}^{k-1}\xi_{t} + \sum_{t=i+1}^{k}\theta_{t} \qquad i=j,\ldots,k-1.\label{eq:xicase2}
\end{align}
\end{subequations}
Proving this claim implies $\xi_{2} > \xi_{3} > \cdots > \xi_{k-1} \ge 1$ since $\theta\ge\onevec$. Equation~\eqref{eq:xi1recurse} gives us $\xi_{1} = (\theta_{1}/\theta_{2})(\sum_{t=2}^{k-1}\xi_{t}-1)$ if $j>1$ or $\xi_{1}=(\theta_{1}/\theta_{2})\sum_{t=2}^{k-1}\xi_{t} + (\sum_{t=2}^{k}\theta_{t})/\theta_{2}$ if $j=1$. For $(j,k)\neq(2,3)$, we then have $\xi_{1}>0$. For $(j,k)=(2,3)$, $\xi_{1}= \theta_{1}(\theta_{3}-1)/\theta_{2}$, which is equal to zero if and only if $\theta_{3}=1$. Thus, proving equations~\eqref{eq:xiformula} finishes our proof for $\phi(\six{j}{k})$.

We prove \eqref{eq:xicase1} and \eqref{eq:xicase2} separately by backward induction on $i$. For \eqref{eq:xicase2}, the base case $\xi_{k-1}=\theta_{k}$ follows by using $\xi_{k}=0$ and  \eqref{eq:vthetadiff} in \eqref{eq:xirecurse2}. Assume \eqref{eq:xicase2} to be true for $j < l \le k-1$ and consider $\xi_{l-1}$. Applying \eqref{eq:xirecurse2} and using the induction hypothesis yields
\begin{align*}
\xi_{l-1} &= \bb_{l-1}\xi_{l} + \xi_{l}  - \theta_{l}\sum_{t=l+1}^{k-1}\xi_{t} +\theta_{l}\\
&= \bb_{l-1}\xi_{l} + \bb_{l}\sum_{t=l+1}^{k-1}\xi_{t} + \sum_{t=l+1}^{k}\theta_{t}  - \theta_{l}\sum_{t=l+1}^{k-1}\xi_{t} +\theta_{l}\\
&= \bb_{l-1}\xi_{l}+ (\bb_{l}-\theta_{l})\sum_{t=l+1}^{k-1}\xi_{t} + \sum_{t=l}^{k}\theta_{t} \\
&= \bb_{l-1}\sum_{t=l}^{k-1}\xi_{t} + \sum_{t=l}^{k}\theta_{t}.
\end{align*} 
For \eqref{eq:xicase1}, the base case formula for $\xi_{j-1}$ can be obtained as follows: from \eqref{eq:xirecurse2} we have
\begin{align*}
\xi_{j-1} &= \bb_{j-1}\xi_{j}+ \xi_{j} - \theta_{j}\sum_{t=j+1}^{k-1}\xi_{t} - \bb_{k}+\theta_{j}\\
&= \bb_{j-1}\xi_{j} + \bb_{j}\sum_{t=j+1}^{k-1}\xi_{t} + \sum_{t=j+1}^{k}\theta_{t} - \theta_{j}\sum_{t=j+1}^{k-1}\xi_{t} - \bb_{k}+\theta_{j}\\
&= \bb_{j-1}\xi_{j} + (\bb_{j}-\theta_{j})\sum_{t=j+1}^{k-1}\xi_{t} + \sum_{t=j}^{k}\theta_{t} - \bb_{k}\\
&= \bb_{j-1}\xi_{j} + \bb_{j-1}\sum_{t=j+1}^{k-1}\xi_{t} - \bb_{j-1}\\
&= \bb_{j-1}\left(\sum_{t=j}^{k-1}\xi_{t} - 1\right).
\end{align*}
The inductive step is similar to that for \eqref{eq:xicase2}. 
\end{proof}

$\P$ is a $\n$-polytope and hence for any $v,v^{\prime}\in\V$, $(v,v^{\prime})$ is an edge in $\G{\P}$ if and only if $|\phi(v)\cap\phi(v^{\prime})| \ge \n-1$ \myred{and there does not exist a $v^{\prime\prime}\in\V$ with $\phi(v^{\prime\prime}) \supseteq \phi(v)\cap\phi(v^{\prime})$}. The formulas for $\phi(\cdot)$ in Proposition~\ref{prop:vfacet} imply a complete list of edges and thereby the degree of each vertex.

\begin{corollary}\label{corr:edges}
If $\theta > \onevec$, $\G{\P}$ has $\frac{1}{2}(d^{2}+d+2)$ vertices, the edges between which are as follows:
\begin{enumerate}
\item $(\theta,\y), (\theta,\six{1}{2})$, and $(\theta,\five{k})$ for $k\ge 3$,
\item $(\zerovec,\y)$ and $(\zerovec,\six{j}{\n})$ for $1 \le j \le\n-1$,
\item $(\y,\six{2}{3})$ if $\d \ge 4$ 
\item $(\y,v)$ for $v\in\V\setminus\{\six{2}{3},\six{1}{\n},\dots,\six{\n-1}{\n}\}$,
\item $(\five{k_{1}},\five{k_{2}})$ for $k_{1},k_{2}\ge 3$,
\item $(\six{k-1}{k},\five{k})$ for $k\ge 3$,
\item $(\six{j}{k_{1}},\five{k_{2}})$ for $2 \le k_{1} \le k_{2}-2$, \myred{$1 \le j \le k_{1}-1$},
\item $(\six{j_{1}}{k},\six{j_{2}}{k})$ for $1\le j_{1},j_{2} \le k-1$,
\item $(\six{j}{k},\six{j}{k+1})$ for $1 \le j < k \le \n-1$.
\newcounter{enumTemp}
    \setcounter{enumTemp}{\theenumi}
\end{enumerate}
If $\theta_k =1$, \myred{$k \geq 3$}, then $\five{k} = \six{k-1}{k}$, and $\G{\P}$ is a minor of the above-described graph obtained by contracting the edges $\{(\five{k},\six{k-1}{k}) \colon \theta_{k}=1\}$.

%
%

\end{corollary}
\begin{proof}
The neighbors of $\zerovec$ and $\theta$ are from Proposition~\ref{prop:neighbors}. \myred{We argue each of the remaining claimed edges to be the only potential edges in $\G{\P}$ by showing that the two vertices forming the edge have at least $\n-1$ facets in common. It is then straightforward to verify that for each such edge $(v,v^\prime)$, no other vertex lies on the common facets for $v$ and $v^{\prime}$, thereby showing that the potential edges are indeed all edges of $\G{\P}$, and thus completing our proof.}

Let us first argue neighbors of $\y$. For $d\ge 4$, $\y$ and $\six{2}{3}$ share $H_{1}$ and $H_{3}$  and at least $\n-3$ planes from $\Hcal$, implying a \myred{possible} edge between the two. For $d=3$, the only coordinate plane they share is $H_{1}$ and so an edge exists  \myred{only if} they share $d-2$ planes from $\Hcal$, which happens only when $\theta_{3}=1$. Since $\Hcal\setminus\{H_{\six{1}{2}},H_{\five{3}},\dots,H_{\five{d}} \} = H_{\y}$, the coordinate vertex $\six{j}{d}$ does not share any plane from $\Hcal$ with $\y$ and shares all the coordinate planes except $H_{j}$ and $H_{d}$. Now $j\le d-1$ implies that $|\phi(\y)\cap\phi(\six{j}{d})|\le d-2$ and so there cannot be an edge between $\y$ and $\six{j}{d}$. For $k\le d-1$, edge $(\y,\six{j}{k})$ \myred{may exist} because $|\phi(\y)\cap\phi(\six{j}{k})|\ge d-1$ due to $\phi(\y)\cap\phi(\six{j}{k})\supseteq \{H_{\five{k+1}},\dots,H_{\five{d}}\} \cup (\Ccal[k]\setminus\{H_{j}\})$. 
Arguments for the edge $(\y,\five{k})$ are similar. The $\five{k}$'s \myred{may} form a clique since for any $3\le k_{1} < k_{2} \le d$, $\phi(\five{k_{1}})\cap\phi(\five{k_{2}})\supseteq \{H_{1}\} \cup (\Hcal\setminus\{H_{\five{k_{1}}},H_{\five{k_{2}}} \})$. 

Consider $\six{j}{k_{1}}$ and $\five{k_{2}}$ for $k_{2}\ge 3, 1 \le j < k_{1}\le d$. We use the following cases.
\begin{description}
\item[$2 \le k_{1}\le k_{2}-1$: ] Here \[\phi(\six{j}{k_{1}})\cap\phi(\five{k_{2}})\cap\Hcal = 
\begin{cases}
\Hcal\setminus\{H_{\five{3}}, H_{\five{k_{2}}} \}& \text{ if } k_{1}=3,\theta_{3}=1\\
\Hcal\setminus\{H_{\six{1}{2}},H_{\five{3}},\dots,H_{\five{k_{1}}},H_{\five{k_{2}}}\} & \text{ $k_{1}=3,\theta_{3}\ge 2$ or $k_{1}\neq 3$. }
\end{cases}\] The cardinality of this set is $d-2$ for $k_{1}=3,\theta_{3}=1$ or $d-k_{1}$ otherwise. Also \[\phi(\six{j}{k_{1}})\cap\phi(\five{k_{2}})\cap\Ccal= \Ccal[k_{1}-1]\setminus\{H_{j}\} \cup \begin{cases}
H_{k_{1}} & \text{ if } k_{1}\le k_{2}-2\\
\emptyset & \text{ if } k_{1}\le k_{2}-1.
\end{cases}
 \] Therefore $|\phi(\six{j}{k_{1}})\cap\phi(\five{k_{2}})|\ge d-1$ if and only if $3 \neq  k_{1}\le k_{2}-2$ or $k_{1}=3,\theta_{3}=1,k_{2}\ge 4$ or $k_{1}=3,\theta_{3}\ge 2,k_{2}\ge 5$.

\item[ $3 \le k_{2}\le k_{1}-1$: ] Here $k_{1}\ge 4$. We have $\phi(\six{j}{k_{1}})\cap\phi(\five{k_{2}})\cap\Hcal = \Hcal\setminus\{H_{\six{1}{2}},H_{\five{3}},\dots,H_{\five{k_{1}}}\}$, which are exactly $d-(k_{1}-1)$ common planes from $\Hcal$. So an edge exists \myred{only if} there  are at least $k_{1}-2$ common coordinate planes. Note that $\phi(\six{j}{k_{1}})\cap\phi(\five{k_{2}})\cap\Ccal \subseteq \{H_{1},\dots,H_{k_{2}-2},H_{k_{2}} \}$ and so we can have at most $k_{2}-1$ common coordinate planes. Hence, for $k_{2}\le k_{1}-2$, there is no edge, and for $k_{1}=k_{2}+1$, an edge exists \myred{only if} $j=k_{2}-1$ and $\theta_{k_{2}}=1$.

\item [$3 \le k_{1}=k_{2}=k$: ] We have $\phi(\six{j}{k})\cap\phi(\five{k})\cap\Hcal = \phi(\six{j}{k})\cap\Hcal$ and so the number of common planes from $\Hcal$ is $d-1$ when $k=3,\theta_{3}=1$ or $d-(k-1)$ otherwise. For $j=k-1$, the first $k-2$ coordinate planes are common, giving us a \myred{potential} edge between $\six{k-1}{k}$ and $\five{k}$ for all $k$. For $1\le j \le k-2$, we get $k-2$ common coordinate planes if and only if $\theta_{k}=1$.
\end{description}

Finally, we argue edges between the $v$ vertices. Each $\six{\cdot}{k}$ component \myred{may be} a clique because $\six{j_{1}}{k}$ and $\six{j_{2}}{k}$ belong to the same $d-(k-1)$ planes from $\Hcal$ and share the $k-2$ coordinate planes in $\Ccal[k]\setminus\{H_{j_{1}},H_{j_{2}} \}$. Now consider $\six{j_{1}}{k_{1}}$ and $\six{j_{2}}{k_{2}}$ with $k_{1} < k_{2}$. At most $k_{1}-1$ coordinate planes are shared and exactly $d-(k_{2}-1)$ planes from $\Hcal$ are shared, making the total number at most $d+k_{1}-k_{2}$. This upper bound is less than $d-1$ if $k_{2}\ge k_{1}+2$, meaning that in this case, no edges exist between the cliques $\six{\cdot}{k_{1}}$ and $\six{\cdot}{k_{2}}$. If $k_{2}=k_{1}+1$, then the upper bound is equal to $d-1$ and is attained if and only if the first $k_{1}-1$ coordinate planes are shared, which happens only when  $j_{1}=j_{2}$.
\end{proof}



\begin{corollary}
For $\theta > \onevec$, the degrees of the vertices of $\G{\P}$ are \[
\begin{split}
&\degree(\theta)=\degree(\zerovec)=\n, \quad \degree(\y) = \frac{\n^{2}-\n+2}{2}, \quad \degree(\five{k}) = \n + \frac{(k-2)(k-3)}{2}, \ \; k\ge 3 \\ 
&\degree(\six{j}{k})=\n, \ \; 2 \le k \le \n, 1 \le j \le k-1.
\end{split}
 \]
 The total number of edges is $\frac{1}{3}(d^{3}+2d)$ and  the average degree is $\frac{2}{3}(d-1 + \frac{d+2}{d^{2}+d+2})$.
\end{corollary}
\begin{proof}
The degree of each vertex follows from the list of edges in Corollary~\ref{corr:edges}. The number of edges is half the sum of all the degrees, making it equal to 
\[
\begin{split}
&\frac{1}{2}\left[ 2d + \frac{\n^{2}-\n+2}{2} + d(d-2) + \sum_{k=3}^{d}\frac{(k-2)(k-3)}{2} + d\sum_{k=2}^{d}(k-1) \right] \\ 
= & \frac{1}{2}\left[2d + \frac{\n^{2}-\n+2}{2} + d^{2}-2d  + \frac{1}{2}\sum_{k=1}^{d-3}k(k+1) + \frac{d^{2}(d-1)}{2} \right] \\
= & \frac{1}{2}\left[\frac{d^{3}+2d^{2}-d+2}{2} + \frac{1}{2}\sum_{k=1}^{d-3}k(k+1) \right]\\
=&\frac{1}{2}\left[\frac{d^{3}+2d^{2}-d+2}{2} +  \frac{1}{2}\left(\frac{(d-3)(d-2)(2d-5)}{6} + \frac{(d-3)(d-2)}{2}\right)  \right] \\
=&\frac{1}{2}\left[\frac{d^{3}+2d^{2}-d+2}{2} +  \frac{(d-3)(d-2)(d-1)}{6}  \right] \\
=&\frac{1}{12}\left[3d^{3}+6d^{2}-3d+6 + d^{3}-6d^{2}+11d-6 \right] \\
=&\frac{d^{3}+2d}{3}.
\end{split}
\]
The average degree is obtained by dividing twice the above number with the number of vertices $(d^{2}+d+2)/2$.
\end{proof}


\begin{corollary} The graph of $\P$ has the following properties.
\begin{enumerate}[{(a)}]
\item The radius of $\G{\P}$ is $r(\G{\P})=2$.
\item The diameter of $\G{\P}$ is \[ d(\G{\P}) = \begin{cases}
3 \; &\n \geq 4\\
2 \; &\n \in \{2,3\}
\end{cases}\]
\item $\G{\P}$ is Hamiltonian.
\item If $\theta > \onevec$, the chromatic number of $\G{\P}$ is $\chi(\G{\P}) = \n$. 
\end{enumerate}
\end{corollary}

\begin{proof}
\begin{enumerate}[{(a)}] 
\item Since the only common neighbor of $\theta$ and $\zerovec$ is $\y$, $r(\G{\P}) \geq 2$. The equality follows from the fact that $\y$ can be chosen as a center: for every  non-neighbor of $\y$ we have $d(\y, \six{j}{\n}) = 2$, because there is a path $\y - \six{j}{\n-1} - \six{j}{\n}$.

\item The distance between the non-neighbors of $\y$ is $d(\six{j_{1}}{\n}, \six{j_{2}}{\n})=1$. Therefore, $d(\G{\P}) \leq 3$. The equality follows from the fact that $\six{1}{\n}$ and $\theta$ have no common neighbors. 

\item Suppose first that $\theta > \onevec$. For each $k$, $3 \leq k \leq \n-1$, let $p_{k}$ be a Hamiltonian path in the clique $\{\six{j}{k} \colon 1 \leq j < k\}$ between $\six{k-2}{k}$ and $\six{k-1}{k}$. 
Then \[\zerovec - \six{1}{\n} - \six{2}{\n} - \cdots - \six{\n -1 }{\n} - \five{\n} - \five{\n-1} -\cdots - \five{3} - \theta - \six{1 }{2}  - p_{3} -p_{4} -\cdots - p_{\n-1} -\y -\zerovec\] is a Hamiltonian cycle in $\G{\P}$. If $\theta_{k} =1$, the same construction gives a Hamiltonian cycle if we take $p_{k}$ be a Hamiltonian path in the clique $\{\six{j}{k} \colon 1 \leq j < k-1\}$.

\item Since $\{\six{j}{\n} \colon 1 \leq j \le \n-1 \}\cup\{\zerovec\}$ is a $d$-clique, $\chi(\G{\P}) \geq \n$. On the other hand, one can readily check that  $\varphi: V \rightarrow \{1,2,\ldots, \n\}$, given by $\varphi(0) = \varphi(\theta)= 1, \varphi(\y) = 2, \varphi(\five{k})=k$, $\varphi(\six{j}{k}) = k-j+1$, is a proper coloring of $\G{\P}$.\qedhere
\end{enumerate}
\end{proof}

\myred{Finally, we compute the edge expansion of $\G{\P}$. Recall that the edge expansion of a graph $G$ on $n$ vertices is defined as
\[{\displaystyle h(G)=\min_{\substack{S\subseteq V(G)\colon\\0<|S|\leq {\frac {n}{2}}}}{\frac {|\partial S|}{|S|}},} \]
where $\partial S:=\{(u,v)\in E(G)\ :\ u\in S,v\in V(G)\setminus S\}$. Notice that $\G{\P}$, having $\mathcal{O}(d^{2})$ vertices, is a relatively sparse graph with average vertex degree $\mathcal{O}(d)$.
 
\begin{theorem} \label{thm:expander}
The edge expansion of the graph $G(\P)$ is $h(G(\P)) =1$.
\end{theorem}}
\begin{proof} For the set $S =\{\six{j}{\n} \colon 1 \leq j \leq \n-1\} \cup \{\zerovec\}$, 
 \[\partial S =  \{(\six{j}{\n}, \six{j}{\n-1}) \colon 1 \leq j \leq \n-2 \} \cup \{(\six{\n-1}{\n}, \five{\n}), (\zerovec, w) \}\] and hence
 \[\frac {|\partial S|}{|S|} = \frac{\n -2 + 2}{\n -1 +1} = 1.\] Therefore,
 \[h(G(\P)) \leq 1.\]
 Suppose now $|S| \leq n/2$. We will consider two cases.

\textbf{Case 1:}  $w \in S$. The set $\partial S$ contains $| S^{c} \cap \N[w]|$ edges incident with $w$. Let $| S^{c} \cap \N[w][\P]^{c}| =a$. Then $0 \leq  a \leq \n -1$ and  $\partial S$ contains $a(\n-1-a)$ edges from the 
 $(\n-1)$-clique $ \N[w][\P]^{c}$. If $0 \leq  a \leq \n-2$, then $a(\n-1-a) \geq a$ and 
 \begin{equation} \label{exp} \frac {|\partial S|}{|S|} \geq \frac {| S^{c} \cap  \N[w][\P]| + a}{|S|} = \frac {| S^{c} \cap  \N[w][\P]| + | S^{c} \cap  \N[w][\P]^{c}|}{|S|} = \frac{|S^{c}|}{|S|} \geq 1.\end{equation} If $a =\n-1$, then $\{\six{j}{\n} \colon 1 \leq j \leq \n-1\} \subseteq S^{c}.$ Let $ k \geq 2$ be the maximal so that \[\{\six{j}{k} \colon 1 \leq j \leq k-1\} \cup \{\five{k+2}, \ldots, \five{\n}\} \not \subseteq S^{c}.\] If such a $k$ doesn't exist then $S \subseteq \{w, \theta, \five{3}, \zerovec\}$ and 
  \[\frac {|\partial S|}{|S|} \geq \frac {| S^{c} \cap  \N[w][\P]|}{|S|} = \frac{|S^{c}| - (\n-1)}{|S|} \geq \frac{n-4 - (\n-1)}{4} \geq 1. \] So, assume such a $k$ does exist and let 
  \[ \left|\left(\{\six{j}{k} \colon 1 \leq j \leq k-1\} \cup \{\five{k+2}, \ldots, \five{\n}\}\right) \cap  S\right| =b \geq 1.\] Then $\partial S$ contains $b(\n-2-b)$ edges from the  $(\n-2)$-clique $\{\six{j}{k} \colon 1 \leq j \leq k-1\} \cup \{\five{k+2}, \ldots, \five{\n}\}$ and, because of the maximality of $k$, $b$ edges of the type $(\six{j}{k}, \six{j}{k+1})$ and $(\five{k'}, \six{k'-1}{k'})$ for some $j <k$ and $k' \geq k+2$. For $1 < b < \n-2$, \[b(\n-2-b)+b \geq \n-1 =a\] and~\eqref{exp} holds. If $b = \n-2$ then $\six{k-1}{k} \in S$, $\six{k}{k+1} \in S^{c}$  and $\partial S$ additionally contains one of the edges $(\six{k-1}{k}, \five{k})$, $(\six{k}{k+1}, \five{k+1})$, $(\five{k}, \five{k+1})$. Then 
  \[b(\n-2-b)+b +1 = \n-1 =a\] and~\eqref{exp} holds. Finally, let $b=1$. Then either $\{\five{2}:=\theta, \five{3}, \five{4}, \ldots, \five{\n}\} \subseteq S^{c}$ or $\partial S$ contains one of the edges $(\six{\n-1}{\n}, \five{\n})$, $(\five{k_{1}}, \five{k_{2}})$ for some $k_{1} \neq k_{2}$. So, suppose $\{\five{2}:=\theta, \five{3}, \five{4}, \ldots, \five{\n}\} \subseteq S^{c}$ and let $\six{j}{k} \in S$. If $j=k-1$ then $(\six{j}{k}, u^{k}) \in \partial S$. If $j < k-1$, then one of the edges $(\six{j}{k}, \six{j}{k-1})$, $(u^{k+1}, \six{j}{k-1})$ is in $\partial S$. Either way, we have established that 
 \[ |\partial S| \geq | S^{c} \cap  \N[w][\P]|  + b(\n-2-b)+b +1 = | S^{c} \cap  \N[w][\P]|  + \n-1 = | S^{c} \cap  \N[w][\P]|  + a\] and~\eqref{exp} holds.

\textbf{Case 2:} $w \notin S$. If $S \subseteq  \N[w][\P]$ then clearly $|\partial S| \geq |S|$. Suppose $|S \cap  \N[w][\P]^{c}| = a \geq 1$. The $a \leq \n-1$ and $\partial S$ contains $a(\n-1-a)$ edges from the  $(\n-1)$-clique $ \N[w][\P]^{c}$. If $1 \leq  a \leq \n-2$, then $a(\n-1-a) \geq a$ and 
 \begin{equation} \label{exp2} \frac {|\partial S|}{|S|} \geq \frac {| S \cap  \N[w][\P]| + a}{|S|} = \frac {| S \cap  \N[w][\P]| + | S \cap  \N[w][\P]^{c}|}{|S|} = \frac{|S|}{|S|} = 1.\end{equation} 
 So, let $a=\n-1$. Then $\{\six{j}{\n} \colon 1 \leq j \leq \n-1\} \subseteq S.$ Let $ k \geq 2$ be the maximal so that \[\{\six{j}{k} \colon 1 \leq j \leq k-1\} \cup \{\five{k+2}, \ldots, \five{\n}\} \not \subseteq S.\] Note that such a $k$ exists because otherwise $|S| \geq n/2$. Let   \[ \left|\left(\{\six{j}{k} \colon 1 \leq j \leq k-1\} \cup \{\five{k+2}, \ldots, \five{\n}\}\right) \cap  S^{c}\right| =b \geq 1.\]  Reasoning as in Case 1, where the role of $S$ and $\partial S$ are swapped, we conclude that $\partial S$ contains at least $a$ more edges and therefore~\eqref{exp2} holds. 
\end{proof}

\section{The grevlex polytope $\Q$}\label{sec:grevlex}
\renewcommand{\V}{\ext(\Q)} 
In this section we will describe the main properties of the polytope $\Q$. Throughout, $\preceq$ will denote the grevlex order.
\subsection{$\fancyV$-polytope}
Consider the following integral points:
\begin{subequations}\label{eq:pointsb}
\begin{eqnarray}
\ub{k} &:=& ( \bb_{k-1} \onevec[k-1], \theta_{k}, \theta_{k+1},\ldots,\theta_{\n}) \qquad 2 \le k\le \n+1 \\
\vb{j}{k} &:=& \myred{( (\bb_{k-1}-1)\onevec[j], 0, \ldots, 0,\theta_{k}+1,\theta_{k+1},\ldots,\theta_{\n})} \qquad 1 \le j < k-1 \le \n,
\end{eqnarray}
\myred{where $\bb_{k}$ is given by \eqref{b-bk}.}
\end{subequations}
In particular, $\ub{2} = \theta$, $\ub{\n+1} = b \onevec[\n]$, and $\vb{j}{\n+1} = (b-1)\onevec[j]$.
By construction, we have
\begin{observation}\label{obs:pink}
$\ub{k}\in\pink$ for all $k$, $\vb{j}{k}\in\pink$ for $k \leq d$, $\vb{j}{\n+1}\notin\pink$ for $1 \leq j \leq \n-1$.
\end{observation}


\begin{proposition}\label{prop:V2}
The vertices of $\Q$ are
\[
\V = \{\zerovec\} \,\bigcup\, \{ \ub{k} \colon 2 \leq k \leq \n +1 \} \,\bigcup\, \{ \vb{j}{k}\colon 1 \leq j < k-1 \leq \n \}.
\]
\end{proposition}

\begin{proof}
It is clear  that $\zerovec$  cannot be written as a nontrivial convex combination of integral points in $\Q$. Suppose $\ub{k} = \sum_{i=1}^s \lambda_i x^i$ is a nontrivial convex combination of some $x^i \in \Q\cap\ints^{d}$. Since $\pink$ is a facet of $\Q$, $x^{i} \in \pink$.  Let \[m = \max\{j \colon x^i_j \neq \theta_j \text{ for some } i \}.\] Then $\theta\lex\x^{i}$ implies $x^i_m \geq \theta_m$, leading to $\sum_{i=1}^s \lambda_i x^i_m > \theta_m$. Therefore, $m=k-1$. Also, $x^i_j = 0$ for all $i$ and $j<k-2$. So, the only possibility is $x^i = \ub{k}$ for all $i$.

Now suppose $\vb{j}{k} = \sum_{i=1}^s \lambda_i x^i$, $k \leq d$ is a nontrivial convex combination of some $x^i \in \Q\cap\ints^{d}$. As before, we conclude $x^{i} \in \pink$. Also, $x^{r}_{l}=0$ for $r \in \{1, 2, \ldots, j-1, j+1, \ldots, k-1\}$. Reasoning the same way as in the case of $\ub{k}$, we conclude that $x^{i}_{r}=\theta_{r}$ for $r >k$ and, therefore, $x^{i}_{k} \geq \theta_{k}+1$. This in turn implies that $x^{i}_{k}=\theta_{k}+1$ and thus $x^{i} = \vb{j}{k}$ for all $i$. \myred{The points  $\vb{j}{\n +1}$, being coordinate vectors, are trivially vertices}.

Now we argue that if $v \in \V\setminus\{\zerovec\}$, then $v$ must be equal to some $\ub{k}$ or $\vb{j}{k}$. Since
\[
\Q = \conv{ \left(\left\{x\in\real^{d}\colon \sum_{i=1}^{d}x_{i} \le \b-1\right\} \, \bigcup \,  \conv{ \left\{x\in\ints^{d}\colon \sum_{i=1}^{d}x_{i} = \b, \theta \lex x \right\}} \right)}
\]
we have
 \[\V \subseteq \ext \{x\in\real^{d}\colon \sum_{i}x_{i} \le \b-1\} \cup \ext(\Q\cap\pink).\] Note that $(\b-1)\onevec[d]$  is not a vertex of $\Q$ because $\zerovec, b\onevec[d] \in \Q$ and we have already shown that all the other  vertices of the simplex $ \{x\in\real^{d}\colon \sum_{i}x_{i} \le \b-1\}$ are also vertices of $\Q$. Suppose now that $\x \in \Q \cap \pink$ is a vertex of $Q$. Let $k=\max\{i\colon \x_{i}\neq\theta_{i} \}$. Then $k \geq 2$ and $\x_{k} \geq \theta_{k}+1$. 
 
 - Suppose first $\x_{k}\geq \theta_{k}+2$ and there is $j < k$ such that $\x_{j}>0$. Let $x' = \x+\onevec[j]-\onevec[k]$ and $x'' = \x-\onevec[j]+\onevec[k]$. Then $\x = (x'+x'')/2$ and since $x', x'' \in \Q$, we conclude $\x$ is not a vertex of $\Q$. If there is no $j < k$ such that $\x_{j}>0$ then $\x_{k} = \bb_{k}$ and $\x = \ub{k+1}$.
 
- Suppose now $\x_{k} =  \theta_{k}+1$. If there exist $ i < j < k$ such that $\x_{i}, \x_{j} >0$ then $\x = \left( (\x +\onevec[i] - \onevec[j]) + (\x - \onevec[i] + \onevec[j]) \right)/2$ and, therefore, $\x$ is not a vertex of $\Q$. Otherwise either $\x = \vb{j}{k}$ for some $j < k-1$ or $\x = ( (\bb_{k-1}-1)\onevec[k-1], \theta_{k}+1,\theta_{k+1},\ldots,\theta_{\n})$. But in the latter case 
 \[\x = \frac{\bb_{k-1}-1}{\bb_{k-1}}   (\bb_{k-1} \onevec[k-1], \theta_{k},\theta_{k+1},\ldots,\theta_{\n}) + \frac{1}{\bb_{k-1}} (\bb_{k}\onevec[k],\theta_{k+1},\ldots,\theta_{\n})\] and, therefore, $\x$ is not a vertex of $\Q$.
 \end{proof}
 
As in Observation~\ref{obs:coord}, the coordinate planes define facets of $\Q$.  We refer to all other facets of $\Q$ as \emph{nontrivial facets}. The grading plane $\pink$ defines a nontrivial facet since the face $\Q \cap \pink$ contains $d$ affinely independent vertices \myred{$\theta = \ub{2},\ub{3}, \ldots,\ub{d+1}$}. We show that the coefficients of any nontrivial facet-defining inequality are nonnegative and nonincreasing.

\begin{lemma}\label{lem:monotone2}
Suppose $\Q \subseteq \{x: cx \le c_{0}\}$ and let $F = \Q \cap\{x: cx = c_{0}\}$ be a nontrivial face of $\Q$. If $F$ is not contained in $x_{i} = 0$ for some $i \in \{1,...,d-1\}$ then $0 \le c_{i+1} \le c_{i}$. Consequently, if $F$ is a nontrivial facet then $0 \le c_{i+1} \le c_{i}$ for all $1\le i\le d-1$.
\end{lemma}
\begin{proof}
 Let $\bar{x}$ be a vertex on $F$ with $\bar{x}_{i} \ge 1$. Consider first $x^{\prime} = \bar{x}- \onevec[i]$. Since $0\le x^{\prime} \grevlex \bar{x}$, we have $cx^{\prime} \le c_{0} = c\bar{x}$, which implies $c_{i+1}\ge 0$. The point $x^{\prime\prime} = \bar{x} - \onevec[i] + \onevec[i+1]$ has the property $0 \le x^{\prime\prime} \grevlex \bar{x}$. Therefore, we have $c x^{\prime\prime} \le c_{0} = c\bar{x}$, which yields $c_{i+1} \le c_{i}$.
\end{proof}

The above property is useful for characterizing the neighbors of $\theta$.


\begin{proposition}\label{prop:neighbors2}
The neighbors of $\theta$ and $\zerovec$ are 
\begin{align*}
\N[\theta][\Q] = \{\ub{3} \}\cup\{\vb{1}{k} \colon 3 \le k \le d+1\},\quad 
\N[\zerovec][\Q] = \{\ub{d+1} \} \cup \{\vb{j}{d+1} \colon 1 \le j \le d-1 \}.
\end{align*}
\end{proposition}
\begin{proof}
Let $F = \Q \cap \{x \colon cx = c_0\}$ be an edge of $\Q$ defined by $cx \le c_{0}$ such that $\theta\in F$. We argue by contradiction that this edge contains exactly one of the proposed neighboring vertices in $\N[\Q]$, or equivalently that none of the other vertices belong to this edge. First suppose that $\vb{j}{k}\in F$ for some $2 \le j < k-1\le d$. Now $c\vb{j-1}{k}- c\vb{j}{k} = (\bb_{k-1}-1)(c_{j-1}-c_{j})$, which is nonnegative because of Lemma~\ref{lem:monotone2}. Therefore $\vb{j-1}{k}\in F$, a contradiction to the edge property of $F$. For $\ub{k}$ with $k\ge 4$,  similar reasoning carries through by considering $\ub{k-1}$. For $\N[\zerovec][\Q]$, noting that the proposed neighbors are coordinate vectors, the proof is exactly the same as that in Proposition~\ref{prop:neighbors}.
\end{proof}

Thus $\Q$ has at least $2d$ facets, $d$ of which are coordinate planes that contain $\zerovec$ and the other $d$ contain $\theta$. We show in Theorem~\ref{thm:dantzig2} that there are no other facets. The following properties about nontrivial facets will be useful.

\begin{lemma}\label{lem:clubspadeu}
Let $F$ be a nontrivial facet of $\Q$ defined by $cx \le c_{0}$. 
\begin{enumerate}
\item $\vb{j}{k}\in F$ implies $\vb{j^{\prime}}{k}\in F$ for all $1\le j^{\prime} \le j$ and $c_{1}=c_{2}=\cdots=c_{j}$.
\item If $F$ is not defined by $\pink$, then $F$ contains some coordinate vertex, i.e., there exists some $1\le j\le d-1$ such that $\vb{j}{d+1}\in F$.
\item $\ub{k}\notin F$ implies $\ub{k^{\prime}}\notin F$ for $k+1\le k^{\prime} \le d+1$.
\end{enumerate}
\end{lemma}
\begin{proof}
(1) This is because $c\vb{j}{t} - c\vb{j^{\prime}}{t} = (c_{j}-c_{j^{\prime}})(\bb_{t-1}-1) \le 0$ and so $c\vb{j}{t}=c_{0}$ implies $c\vb{j^{\prime}}{t} = c_{0}$. Consequently, we also get $c_{1}=c_{2}=\cdots=c_{j}$.

(2) To argue this, recall from Observation~\ref{obs:pink} that $\V\setminus\{\zerovec,\vb{1}{d+1},\dots,\vb{d}{d+1}\}\subseteq \pink$. So if $\vb{j}{d+1}\notin F$ for all $j$, the nontriviality of $F$ (i.e., $\zerovec\notin F$) then implies that $F=\Q\cap\pink$.

(3) We know from Lemma~\ref{lem:monotone2} that $F=\{x\in\Q\colon \alpha x = \alpha_{0}\}$ with $0\le \alpha_{n} \le \alpha_{n-1} \le \cdots \le \alpha_{1}$. Then for $2\le t\le d$, we have \[\alpha\ub{t} - \alpha\ub{t+1} = \alpha_{t-1}\bb_{t-1} + \alpha_{t}(\theta_{t} - \bb_{t}) = (\alpha_{t-1} - \alpha_{t})\bb_{t-1} \ge 0. \] Therefore, $\alpha\ub{d+1} \le \alpha\ub{d} \le \cdots \le \alpha\ub{2}=\alpha\theta$. If $\ub{k}\notin F$, then $\alpha\ub{k}<\alpha_{0}$ and the claim follows.
\end{proof}

\myred{\begin{theorem}\label{thm:dantzig2}
$\Dant{\Q}{\zerovec}{\theta}$. \end{theorem}
}
\begin{proof}
To show that $\Dant{\Q}{\zerovec}{\theta}$, we need to prove that every facet of $\Q$ contains either $\zerovec$ or $\theta$. Let $F$ be a facet  of $\Q$ induced by the valid inequality $c x \le c_0$.  If $F$ doesn't contain $\theta$ nor any of the vertices $\vb{1}{k}$, then it is contained in the subspace $x_{1}=0$ and hence is equal to the trivial facet defined by $x_{1}\ge 0$ and therefore contains $\zerovec$. For an arbitrary nontrivial facet $F$, Lemma~\ref{lem:monotone2} tells us 
\begin{subequations}
\begin{equation}\label{eq:monoc}
0\le c_{n} \le c_{n-1}\le \cdots \le c_{1},
\end{equation}

Now assume $\vb{1}{k}\in F$ for some $3 \le k\le d-1$ and $\zerovec \notin F$. Then it must be that $F$ is a nontrivial facet. Suppose for contradiction that $\theta\notin F$. This means $F\neq\Q\cap\pink$ and also has two other implications. First, we have $c_{1}>c_{k-1}$.  Suppose this is not true, which by \eqref{eq:monoc} means that $c_{1}=c_{2}=\dots=c_{k-1}$. Then \[0 < c_{0} - c\theta = c\vb{1}{k} - c\theta = c_{k} + c_{1}(\bb_{k-1}-1) - \sum_{j=1}^{k-1}c_{1}\theta_{j} = c_{k}-c_{1},\] which contradicts $c_{1}\ge c_{k}$ from \eqref{eq:monoc}. Second, we have $\ub{k}\notin F$ for all $k$ due to $\theta=\ub{2}$, the assumption  $c\theta < c_{0}$ and the third item in Lemma~\ref{lem:clubspadeu}.

Now let $j$ be the maximal index such that $\vb{j}{d+1}\in F$; we know such a $j$ exists because of $F\neq\Q\cap\pink$ and the second item in Lemma~\ref{lem:clubspadeu}. If $j\ge k-1$, then applying the first item in Lemma~\ref{lem:clubspadeu} to $\vb{j}{d+1}$ would imply $c_{1}=c_{k-1}$, a contradiction to $c_{1}>c_{k-1}$. Hence $1 \le j \le k-2$ and \eqref{eq:monoc} and maximality of $j$ lead us to
\begin{equation}\label{eq:monoc2}
c_{1}=\cdots=c_{j}, \qquad c_{1}>c_{i}, \quad i=j+1,\dots,d.
\end{equation}
Using above and $c\theta < c\vb{1}{k}$ by assumption, gives us 
\begin{equation}\label{eq:contra1}
c_{k} - c_{1} > \sum_{i=j+1}^{k-1}(c_{i}-c_{1})\theta_{i}.
\end{equation}

Since $F$ is not contained in any coordinate plane and $\ub{k}\notin F \ \forall k$, we know that for every $t=j+1,\dots,k-1$, there exist some $(i_{t},k_{t})$ such that $\vb{i_{t}}{k_{t}}\in F$ and either $k_{t}\le t$ or $i_{t}=t$. The second possibility $i_{t}=t$ can be ruled out since applying the first item in Lemma~\ref{lem:clubspadeu} to $\vb{t}{k_{t}}$ would imply $c_{1}=c_{t}$, a contradiction to \eqref{eq:monoc2} due to $t \ge j+1$. Therefore $1\le i_{t} < k_{t}-1 \le t-1$. Now since $\vb{i_{t}}{k_{t}}\in F$, the first item in Lemma~\ref{lem:clubspadeu} implies $\vb{1}{k_{t}}\in F$. Therefore we have $c\vb{1}{k_{t}} = c\vb{1}{k}$, which upon simplification yields $c_{k} - c_{k_{t}} = \sum_{i=k_{t}}^{k-1}(c_{i}-c_{1})\theta_{i}$ for $j+1 \le t \le k-1$. Choosing $t=j+1$  gives us 
\begin{equation}\label{eq:contra2}
c_{k} - c_{k_{j+1}} = \sum_{i=k_{j+1}}^{k-1}(c_{i}-c_{1})\theta_{i} = \sum_{i=j+1}^{k-1}(c_{i}-c_{1})\theta_{i},
\end{equation}
\end{subequations}
where the second equality is due to $k_{j+1}\le j+1$ by construction, and $c_{1}=\cdots=c_{j}$ from \eqref{eq:monoc2}.  Since $k_{j+1} \le j+1$,  \eqref{eq:monoc2} tells us $c_{k_{j+1}} \le c_{1}$. Substituting this into \eqref{eq:contra2} leads to $c_{k} - c_{1} \le  \sum_{i=j+1}^{k-1}(c_{i}-c_{1})\theta_{i}$, but this is a contradiction to \eqref{eq:contra1}.
\end{proof}

Similar to $\P$, for $\n \geq 4$ the only antipodal vertex pair of  $\Q$ is $(\zerovec,\theta)$. We will show this in Corollary~\ref{corr:Qantipode}.

\subsection{$\fancyH$-polytope}


By Proposition~\ref{prop:dantcone}, we need to invert
\begin{eqnarray*}
\myred{\bar{M}} &=& 
\left[\begin{array}{cccccc}\ub{3} - \theta & \vb{1}{3} - \theta & \vb{1}{4} - \theta & \cdots & \vb{1}{d}-\theta & \vb{1}{d+1}-\theta\end{array}\right] \medskip\\
& = & 
{\small \left[\begin{array}{ccccccc}
-\theta_1 & \theta_2 -1 & \theta_2+\theta_{3}-1 & \theta_2+\theta_{3}+\theta_{4}-1 & \cdots & \theta_{2}+\cdots + \theta_{d-1}-1 & \theta_2+\cdots + \theta_{d}-1 \\
\theta_1 &  -\theta_{2} & -\theta_2 & -\theta_2 & \cdots & -\theta_{2} & -\theta_2 \\
0 & 1 & -\theta_{3} & -\theta_3 & \cdots & -\theta_{3} & -\theta_3 \\
\vdots & 0 & 1 & -\theta_{4} & \cdots & -\theta_{4} & -\theta_4 \\
\vdots & \vdots & 0 & 1 & \ddots & \vdots & \vdots \\
\vdots & \vdots & \vdots & \vdots & \ddots & -\theta_{d-1} & \vdots \\
0 & \cdots & \cdots & \cdots & \cdots & 1 &  -\theta_{d} 
\end{array}\right] }
\end{eqnarray*}

Let 
\[ q_{i}^{j}= \begin{cases} \theta_{i}  \theta_{j}\prod_{k=i+1}^{j-1} (\theta_{k}+1)& j >i \medskip\\
                                         \theta_{i} & j=i \medskip \\
                                         1 & j < i.
                  \end{cases}
 \]

\begin{proposition}\label{prop:N2}
$\Q = \left\{x \ge \zerovec\mid \myred{\bar{N}}x \ge \myred{\bar{N}}\theta \right\}$ where $\myred{\bar{N}}=\myred{\bar{M}}^{-1}$ with

\[
\myred{\bar{N}}_{i,\n}=\begin{cases}\displaystyle
-1 & i=\n \medskip\\
-\theta_{\n}+1 & i=\n-1 \medskip\\
-q_{i+1}^{\n} & 2 \leq i \leq \n-2 \medskip\\
-\frac{q_{2}^{\n}}{\theta_{1}}  & i=1
\end{cases} \qquad 
\myred{\bar{N}}_{1,j}=\myred{\bar{N}}_{1,j+1}+\begin{cases}\displaystyle
-\frac{1}{\theta_{1}} & j=1\medskip\\
-\frac{\theta_{2}-1}{\theta_{1}} &  j=2\medskip\\
-\frac{q_{2}^{j}}{\theta_{1}}  &  3 \leq j \leq \n-1
\end{cases}
\]
and for $i \geq 2$
\[
\myred{\bar{N}}_{i,j}=\myred{\bar{N}}_{i,j+1}+\begin{cases}\displaystyle
0 & j<i \medskip\\
-1 & j=i\medskip\\
-\theta_{j}+1 &  j=i+1\medskip\\
-q_{i+1}^{j}  &  i+2 \leq j \leq \n-1.
\end{cases}
\]
\end{proposition}
\begin{proof}

Suppose $i \neq 1, j \neq 1$. Then
\begin{align}
\myred{\bar{N}}_{i\cdot} \myred{\bar{M}}_{\cdot j} &= \sum_{l=1}^{j+1}\myred{\bar{N}}_{i,l}\myred{\bar{M}}_{l, j+1} \nonumber \\
&= \myred{\bar{N}}_{i,1}(\theta_{2} + \cdots +\theta_{j}-1) - \sum_{l=2}^{j}\myred{\bar{N}}_{i,l}\theta_{j} +\myred{\bar{N}}_{i, j+1} \nonumber \\
&= \sum_{l=2}^{j} (\myred{\bar{N}}_{i,1} - \myred{\bar{N}}_{i,l})\theta_{l} + (\myred{\bar{N}}_{i,j+1}-\myred{\bar{N}}_{i,1}) \label{star}
\end{align}

So, if $j < i$, then since $\myred{\bar{N}}_{i,1} = \myred{\bar{N}}_{i,2}=\cdots =\myred{\bar{N}}_{i,i}$ we have $\myred{\bar{N}}_{i\cdot} \myred{\bar{M}}_{\cdot j}=0$. If $j=i$, then $\myred{\bar{N}}_{i\cdot} \myred{\bar{M}}_{\cdot j}=\myred{\bar{N}}_{i,i+1}-\myred{\bar{N}}_{i,i}=1$. For $j >i$,
\[\myred{\bar{N}}_{i\cdot} \myred{\bar{M}}_{\cdot j}= \myred{\bar{N}}_{i\cdot} \myred{\bar{M}}_{\cdot j-1} + (\myred{\bar{N}}_{i,1}-\myred{\bar{N}}_{i,j})\theta_{j} + (\myred{\bar{N}}_{i,j+1}-\myred{\bar{N}}_{i,j}).\]
Therefore, if $j=1+1$, then
\[\myred{\bar{N}}_{i\cdot} \myred{\bar{M}}_{\cdot i+1} = 1 + (-1)\theta_{i+1} + (\theta_{i+1} -1) =0.\]
For, $j > i+1$, inductively we get
\[\myred{\bar{N}}_{i\cdot} \myred{\bar{M}}_{\cdot j} = 0 + (-\sum_{k=i+1}^{j-1}q_{i+1}^{k}) \theta_{j} +q_{i+1}^{j} =0.\]

The entries $\myred{\bar{N}}_{1\cdot} \myred{\bar{M}}_{\cdot j}$ and $\myred{\bar{N}}_{i\cdot} \myred{\bar{M}}_{\cdot 1}$ can be computed in a similar way.
\end{proof}

\subsection{Graph of the polytope}\label{sec:G2}
\renewcommand{\phi}{\psi_{\Q}}

We derive some basic properties of the graph of $\Q$, denoted by $\G{\Q}$. This graph has $\frac{\n^{2}+\n+2}{2}$ vertices enumerated in Proposition~\ref{prop:V2}. To find all the edges of $\G{\Q}$, we characterize the vertex-facet incidence for $\Q$ in Proposition~\ref{prop:vfacet2}. Adopting the same approach as in \textsection\ref{sec:G} to denote $H_{v}$, for $v\in\N[\theta][\Q]$, as the only facet-defining hyperplane that contains $\theta$ but not $v$, we have 
\begin{align*}
\Hcal[\Q] &:= \{H_{\ub{3}}, H_{\vb{1}{3}},\ldots,H_{\vb{1}{d+1}}\}, \text{ with }\\
H_{\ub{3}} &= \{x\mid \myred{\bar{N}}_{1\cdot}(x-\theta)=0\}, \qquad H_{\vb{1}{k}} = \{x\mid \myred{\bar{N}}_{(k-1)\cdot}(x-\theta)=0\} \quad 3 \le k \le d+1,
\end{align*}
where $N$ is the matrix inverse from Proposition~\ref{prop:N2}. The hyperplanes incident to $\zerovec$ are the coordinate planes denoted in \eqref{eq:Hzero}. As before, for any $v\in\V$, let $\phi(v)$ denote the subset of facet-defining hyperplanes of $\Q$ that contain $v$.


\begin{proposition}\label{prop:vfacet2}
We have
\begin{align*}
\phi(\zerovec) &= \Ccal,\quad \phi(\theta) = \Hcal[\Q],\\
\phi(\ub{k}) &= (\Hcal[\Q]\setminus\{H_{\ub{3}}, H_{\vb{1}{3}}, \dots, H_{\vb{1}{k-1}}\}) \cup \Ccal[k-2], \qquad 3 \le k \le d+1\\
\phi(\vb{j}{k}) &= (\Hcal[\Q]\setminus\{H_{\ub{3}}, H_{\vb{1}{3}}, \dots, H_{\vb{1}{j}}, H_{\vb{1}{k}} \}) \cup (\Ccal[k-1]\setminus\{H_{j} \}), \qquad 2 \le j < k-1 \le d\\
\phi(\vb{1}{k}) &= (\Hcal[\Q]\setminus\{H_{\vb{1}{k}} \}) \cup (\Ccal[k-1]\setminus\{H_{1} \}), \qquad 3 \le k \le d+1.
\end{align*}
\end{proposition}
\begin{proof}
The coordinate planes are trivial to check due to $\theta\ge\onevec$, whereas $\phi(\zerovec)$ and $\phi(\theta)$ follow from the construction of $\Hcal[\Q]$ and $\Ccal$. It remains to argue the incidence of a hyperplane in $\Hcal[\Q]$ onto a $\ub{k}$, for $k\ge 3$, or a $\vb{j}{k}$. The formula for $\myred{\bar{N}}$ in Proposition~\ref{prop:N2} and the monotone property of facet coefficients in Lemma~\ref{lem:monotone2} tells us that 
\begin{equation}\label{eq:Nrow}
-\myred{\bar{N}}_{t-1,1} = \cdots = -\myred{\bar{N}}_{t-1,t-1} > -\myred{\bar{N}}_{t-1,t} \ge -\myred{\bar{N}}_{t-1,t+1} \ge \cdots \ge -\myred{\bar{N}}_{t-1,d}  \quad 2 \le t \le d. 
\end{equation} 
Consider $\ub{k}$ for $k\ge 3$. Since $\ub{3}\notin H_{\ub{3}}$ by construction, the third item in Lemma~\ref{lem:clubspadeu} implies that $\ub{k}\notin H_{\ub{3}}$. For any $3\le j \le k-1$, since  $\theta\in H_{\vb{1}{j}}$, we have $\ub{k}\notin H_{\vb{1}{j}}$ if and only if $\myred{\bar{N}}_{(j-1)\cdot}\ub{k} < \myred{\bar{N}}_{(j-1)\cdot}\theta$. Now 
\begin{align*}
\myred{\bar{N}}_{(j-1)\cdot}\ub{k} &= \myred{\bar{N}}_{j-1,k-1}\bb_{k-1} + \sum_{i=k}^{d}\myred{\bar{N}}_{j-1,i}\theta_{i} \\
&= \underbrace{\myred{\bar{N}}_{j-1,k-1}\bb_{j-1}}_{<\, \myred{\bar{N}}_{j-1,1}\bb_{j-1}} \:+\: \underbrace{\myred{\bar{N}}_{j-1,k-1}\sum_{i=j}^{k-1}\theta_{i}}_{\le\, \sum_{i=j}^{k-1}\myred{\bar{N}}_{j-1,i}\theta_{i}} \:+\: \sum_{i=k}^{d}\myred{\bar{N}}_{j-1,i}\theta_{i} \\&< \myred{\bar{N}}_{(j-1)\cdot}\theta,
\end{align*}
where the inequalities are due to equation~\eqref{eq:Nrow}. Therefore $\ub{k}\notin H_{\vb{1}{j}}$ for $3\le j \le k-1$, and so $\ub{k}\in H$ for some $H\in\Hcal[\Q]$ only if $H$ is one of the $d-k+2$ hyperplanes $H_{\vb{1}{k}}, H_{\vb{1}{k+1}}, \dots, H_{\vb{1}{d+1}}$. Since exactly $k-2$ coordinate planes contain $\ub{k}$ and we know that $|\phi(\ub{k})|\ge d$ due to $\ub{k}$ being a vertex of the $d$-polytope $\Q$, it follows that $\ub{k}\in H_{\vb{1}{t}}$ for $k \le t \le d+1$. 

Now we derive $\phi(\vb{j}{k})$. Note that $k\ge 3$. The construction of $\Hcal[\Q]$ and $\vb{1}{k}\in\N$   implies that $\phi(\vb{1}{k}) \cap\Hcal[\Q] = \Hcal[\Q]\setminus\{H_{\vb{1}{k}} \}$. This leads to $\vb{j}{k}\notin H_{\vb{1}{k}}$  because otherwise the first item in Lemma~\ref{lem:clubspadeu} gives  the contradiction $\vb{1}{k}\in H_{\vb{1}{k}}$. Consider the hyperplane $H_{\vb{1}{t}} := \{x\mid \myred{\bar{N}}_{(t-1)\cdot}x=c_{0}\}$ for $3\le t \le d+1, t \neq k$, which contains $\vb{1}{k}$. Then $\vb{j}{k}\in H_{\vb{1}{t}}$ if and only if $\myred{\bar{N}}_{(t-1)\cdot}\vb{j}{k} = \myred{\bar{N}}_{(t-1)\cdot}\vb{1}{k}$.  Now, $\myred{\bar{N}}_{(t-1)\cdot}\vb{j}{k} - \myred{\bar{N}}_{(t-1)\cdot}\vb{1}{k} = (\myred{\bar{N}}_{t-1,j}-\myred{\bar{N}}_{t-1,1})(\bb_{k-1}-1)$ and since $\bb_{k-1} > 1$ for $k\ge 3$ due to $\theta\ge\onevec$, we have $\vb{j}{k}\in H_{\vb{1}{t}}$ if and only if $\myred{\bar{N}}_{t-1,j} = \myred{\bar{N}}_{t-1,1}$. Equation~\eqref{eq:Nrow} tells us that $\myred{\bar{N}}_{t-1,1}=\myred{\bar{N}}_{t-1,j}$ if and only if $j \le t-1$, which, along with $t\neq k$, is equivalent to $t \in \{j+1,\dots,k-1,k+1,\dots,d+1\}$. The claim for $\phi(\vb{j}{k})$ follows. The arguments for $\vb{j}{k}\notin H_{\ub{3}}$, for $j\ge 2$, are similar.
\end{proof}

Since $(v,v^{\prime})$ is an edge in $\G{\Q}$ if and only if $|\phi(v)\cap\phi(v^{\prime})| \ge \n-1$, Proposition~\ref{prop:vfacet2} implies a complete list of edges  \myred{(Figure~\ref{fig:graphGofQ})} and  thereby the degree of each vertex.

\begin{corollary} 
$\G{\Q}$ has $\frac{1}{2}(d^{2}+d+2)$ vertices, the edges between which are as follows:
\begin{enumerate}
\item $(\zerovec, \ub{\n+1})$ and $(\zerovec, \vb{j}{\n+1})$ for $1 \leq j \leq \n-1$,
\item $(\ub{k}, \ub{k+1})$ for $2 \leq k \leq \n$,
\item $(\ub{k}, \vb{j}{k-1})$ for \myred{$4 \leq k \leq \n+1$, $1 \leq j \leq k-3$},
\item $(\ub{j}, \vb{j-1}{k})$ for $2 \leq j \leq \n$, \myred{$j+1 \leq k \leq \n+1$},
\item $(\vb{j}{k_{1}}, \vb{j}{k_{2}})$ for  \myred{$3\leq j+2\leq k_{1} < k_{2} \leq \n+1$}, 
\item $(\vb{j_{1}}{k}, \vb{j_{2}}{k})$ for \myred{$1 \leq j_{1} <  j_{2} \leq k-2 \leq \n-1$}.
\end{enumerate}
\end{corollary}

\myred{\begin{proof}  Based on Proposition~\ref{prop:vfacet2}, one can check that $|\phi(v)\cap\phi(v^{\prime})| \ge \n-1$ for the pairs of vertices $(v,v^{\prime})$ given in items 1-6 as well as  the pairs $(\vb{j_{1}}{k_{1}}, \vb{j_{2}}{k_{2}})$ in which ($j_{1} < j_{2} < k_{1} < k_{2}$ and $k_{1} \geq j_{2}+3$) or ($j_{2} < j_{1} < k_{1} < k_{2}$ and $k_{1} \geq j_{1}+3$). For the pairs $(v,v^{\prime})$ from 1-6, it is straightforward to verify that no other vertex lies on the common facets for $v$ and $v'$. However, in the last case, when $j_{1} < j_{2} < k_{1} < k_{2}$ and $k_{1} \geq j_{2}+3$, we have also $\phi(\vb{j_{1}}{k_{1}})\cap\phi(\vb{j_{2}}{k_{2}}) \subseteq \phi(\vb{j_{2}}{k_{1}})$. Similarly, when $j_{2} < j_{1} < k_{1} < k_{2}$ and $k_{1} \geq j_{1}+3$, we have  $\phi(\vb{j_{1}}{k_{1}})\cap\phi(\vb{j_{2}}{k_{2}}) \subseteq \phi(\vb{j_{1}}{k_{2}})$. Therefore, none of these pairs  determines an edge of $\Q$.
\end{proof}
}

\begin{figure}[ht]

\centering
 
\resizebox{11cm}{!}{

\begin{tikzpicture}[line width=1pt]


\sffamily

\node (u2){$u^{2}$};
\node[right=1cm of u2] (u2r){};
\node[right=2cm of u2] (v13){$v^{1,3}$};
\node[right=1cm of v13] (v14){$v^{1,4}$};
\node[right=1.5cm of v14] (1dots){$\cdots$};
\node[right=1.5cm of 1dots] (v1dm1){$v^{1,d-1}$};
\node[right=1cm of v1dm1] (v1d){$v^{1,d}$};
\node[right=1cm of v1d] (v1dp1){$v^{1,d+1}$};

\node[below=1cm of u2r](u3){$u^{3}$};
\node[right=1cm of u3] (u3r){};
\node[below=1cm of v14] (v24){$v^{2,4}$};
\node[right=1.5cm of v24] (2dots){$\cdots$};
\node[right=1.5cm of 2dots] (v2dm1){$v^{2,d-1}$};
\node[right=1cm of v2dm1] (v2d){$v^{2,d}$};
\node[right=1cm of v2d] (v2dp1){$v^{2,d+1}$};

\node[below=1cm of u3r](u4){$u^{4}$};
\node[right=1cm of u4] (u4r){};
\node[below=1cm of 2dots] (3dots){$\cdots$};
\node[right=1.5cm of 3dots] (v3dm1){$v^{3,d-1}$};
\node[right=1cm of v3dm1] (v3d){$v^{3,d}$};
\node[right=1cm of v3d] (v3dp1){$v^{3,d+1}$};

\node[below=1cm of u4r](u5){$u^{5}$};
\node[below=1cm of 3dots] (4dots){$\cdots$};
\node[right=1.5cm of 4dots] (v4dm1){$v^{4,d-1}$};
\node[right=1cm of v4dm1] (v4d){$v^{4,d}$};
\node[right=1cm of v4d] (v4dp1){$v^{4,d+1}$};

\node[below=1.5cm of v4dm1](1vdots){$\vdots$};
\node[below=1.5cm of v4d](2vdots){$\vdots$};
\node[below=1.5cm of v4dp1](3vdots){$\vdots$};

\node[below=3cm of v4dm1](ud){$u^{d}$};
\node[right=1cm of ud](udr){};
\node[below=3cm of v4dp1] (vdm1dp1){$v^{d-1,d+1}$};

\node[below=4.5cm of v4d] (udp1){$u^{d+1}$};
\node[below=5.5cm of v4dp1](z){$0$};

\node[below=1cm of 4dots](ddots){$\ddots$};

\node[shape=ellipse,draw=myblue,minimum size=1cm,fit={(u2) (v1dp1)}] (el1h){};
\node[shape=ellipse,draw=myblue,minimum size=1cm,fit={(u3) (v2dp1)}] (el2h){};
\node[shape=ellipse,draw=myblue,minimum size=1cm,fit={(u4) (v3dp1)}] (el3h){};
\node[shape=ellipse,draw=myblue,minimum size=1cm,fit={(u5) (v4dp1)}] (el1h){};
\node[shape=ellipse,draw=myblue,minimum size=1cm,fit={(ud) (vdm1dp1)}] (eldh){};

\node[shape=ellipse,draw=myblue,minimum size=1cm,fit={(u4) (v13)}] (el1v){};
\node[shape=ellipse,draw=myblue,minimum size=1cm,fit={(u5) (v14)}] (el2v){};
\node[shape=ellipse,draw=myblue,minimum size=1cm,fit={(ud) (v1dm1)}] (eldv){};
\node[shape=ellipse,draw=myblue,minimum size=1cm,fit={(udp1) (v1d)}] (eldp1v){};
\node[shape=ellipse,draw=myblue,minimum size=1cm,fit={(z) (v1dp1)}] (eldp1v){};

\draw [bend right=60,myblue, ultra thick] (u2) to (u3);
\draw [bend right=60,myblue, ultra thick] (u3) to (u4);
\draw [bend right=60,myblue, ultra thick] (u4) to (u5);

\node[right=1cm of u5](u5r){};
\node[below=1cm of u5r](u6){};

\draw [bend right=60,myblue, ultra thick] (u5) to (u6);

\node[above=1cm of ud](uda){};
\node[left=1cm of uda](udm1){};

\draw [bend right=60,myblue, ultra thick] (udm1) to (ud);

\draw [bend right=60,myblue, ultra thick] (ud) to (udp1);

\draw [bend right=60,myblue, ultra thick] (udp1) to (z);

\end{tikzpicture}
}
\caption{The graph $\G\Q$ for a vertex $\theta \geq \onevec$ in $\mathbb{R}^{d}$. The circled vertices form cliques.}
\label{fig:graphGofQ}
\end{figure}
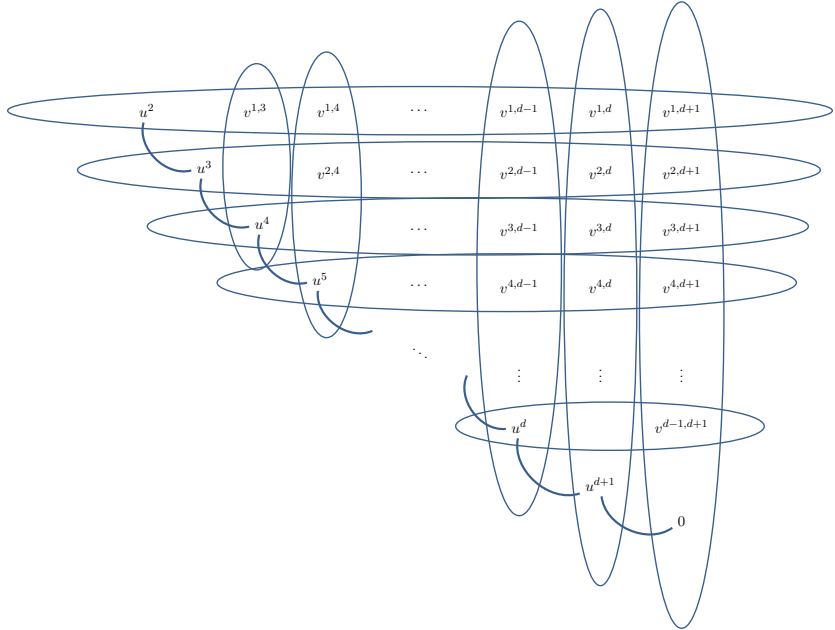

\begin{corollary}\label{corr:Qdegree}
The degrees of the vertices of $\G{\Q}$ are \[ \degree(\zerovec) = \degree(\ub{k}) =\n, \quad \myred{\degree(\vb{j}{k}) = \n+k-j-2}.\] \myred{The total number of edges is $\frac{1}{3}(d^{3}+2d)$ and  the average degree is $\frac{2}{3}(d-1 + \frac{d+2}{d^{2}+d+2})$}.
\end{corollary}

As a consequence we see that $\P$ and $\Q$ define two different families of Dantzig figures. 

\myred{
\begin{corollary}\label{corr:equiv}
Let $\n\ge 3$ be fixed.
\begin{enumerate}
\item For any $\P$ and $\P^{\prime}$ corresponding to $\theta,\theta^{\prime}>\onevec$, we have $\P\cong\P^{\prime}$.
\item For any $\Q$ and $\Q^{\prime}$ corresponding to $\theta,\theta^{\prime}\ge\onevec$, we have $\Q\cong\Q^{\prime}$.
\item For $\theta > \onevec$, $\P\not\cong\Q$.
\end{enumerate}
\end{corollary}
}

\begin{proof}
The first two claims follow from Propositions~\ref{prop:vfacet} and~\ref{prop:vfacet2}. The fact that $\P$ and $\Q$ are not combinatorially equivalent is also not hard to see from the properties we have proved so far.  For $d=3$, as can be seen from Figure~\ref{fig:bliblablup}, $\G{\P}$ has a pentagonal facet, while $\G{\Q}$ doesn't. For $d\geq 4$, the highest degree vertex in $\P$ is $w$ with $\deg(w) = \frac{d^{2}-d+2}{2}$, while in $\Q$, the highest degree vertex is $\vb{1}{\n+1}$ with $\deg(\vb{1}{\n+1})= \myred{2\n-2} < \frac{d^{2}-d+2}{2}$.
\end{proof}

In $\n=3$, $\Dant{\Q}{\vb{1}{3}}{\vb{2}{4}}$. But, for $\n \geq 4$ $(\zerovec,\theta)$ are the only antipodal vertices of $\Q$.

\begin{corollary}\label{corr:Qantipode}
For $\n \geq 4$, $(\zerovec,\theta)$ is the only antipodal vertex pair that generates the Dantzig figure $\Q$. For $\n=3$, $(\vb{1}{3},\vb{2}{4})$ is the other antipodal vertex pair.

\end{corollary}
\begin{proof}
The case $\n=3$ can easily be analyzed from Figure~\ref{fig:bliblablup}. Let $\n \geq 4$. Since any antipodal vertex pair $(x,y)$ of a $d$-dimensional Dantzig figure must have $\degree(x)=\degree(y)=\n$, Corollary~\ref{corr:Qdegree} tells us that the only candidate vertices for forming an antipodal pair of $\Q$ are $\zerovec,\theta, \{\ub{k}\colon 3 \le k \le d+1\}, \{\vb{j}{j+2}\colon 1 \le j \le \n-1 \}$. We also need $\phi(x)\cap\phi(y) =\emptyset$ for any antipodal pair. Proposition~\ref{prop:vfacet2} gives us $H_{1}\in\phi(\ub{k})\cap\phi(\vb{j}{j+2})$ for $k\ge 3, j \ge 2$, $\phi(\ub{k})\cap\Hcal[\Q]\neq\emptyset$ for $k\ge 3$,  and $\phi(\vb{j}{j+2})\cap\Hcal[\Q]\neq\emptyset$ for $j\ge 1$. The only remaining possibility is $\vb{1}{3}$ but this is also easy to discard with similar arguments.
\end{proof}

\begin{corollary} The graph of $\Q$ has the following properties.
\begin{enumerate}[{(a)}]
\item The radius of $\G{\Q}$ is $r(\G{\Q})=2$.
\item The diameter of $\G{\Q}$ is $d(\G{\Q})=2$.
\item $\G{\Q}$ is Hamiltonian.
\item The chromatic number of $\G{\Q}$ is $\chi(\G{\Q}) \myred{=} \n$.
\end{enumerate}
\end{corollary}

\begin{proof} \myred{Due to the grid-like structure of $\G{\Q}$ illustrated in Figure~\ref{fig:graphGofQ}, one can easily see that the distance between any two vertices is at most $2$.  Therefore, $r(\G{\Q})=d(\G{\Q})=2$. This also allows for an easy construction of a Hamiltonian cycle, for example, one can start with $0 - u^{d+1} - u^{d} - \cdots - u^{2}$ and then start traversing the cliques depicted with vertical ellipses in Figure~\ref{fig:graphGofQ} from left to right, before coming back to 0. Finally, $\chi(\G{\Q}) = \n$ because $\G{\Q}$ has a $\n$-clique and one possible proper coloring with $\n$ colors is given by: $\varphi(0)=1$, $\varphi(\vb{j}{k}) = k+j (\mathrm{mod} \;{\n})$, $1 \leq j \leq k-2 \leq \n-1$, $\varphi(\ub{k}) = 2k-1 (\mathrm{mod} \;{\n})$, $2 \leq k \leq \n$, $\varphi(\ub{\n +1}) = 0$.}
%
%
%
%
\end{proof}

\section*{Acknowledgements}
\myred{The authors thank Wayne Goddard for computing the expansion number of $\G\P$ for a few examples. They also thank the two anonymous referees for their comments that led to improvement of the paper and especially for catching a mistake in one of the arguments.}  A.G. is partially supported by ONR grant N00014-16-1-2725. S.P. is partially supported by NSF grant DMS-1312817.

\bibliographystyle{abbrvnat}
\bibliography{../Revision/Grevlex}

\end{document}